\documentclass{article}
\input{macros-arXiv.sty}

\usepackage{url}            
\usepackage{booktabs}       
\usepackage{amsfonts, bm, amssymb}       
\usepackage{nicefrac}       
\usepackage{microtype}      
\usepackage{tikz} 
\usepackage{float}

\usepackage{array}
\usepackage{diagbox}
\usepackage{mathtools}
\usepackage{multirow, makecell}
\usepackage{rotating}
\usepackage{graphicx}
\usepackage{subfigure}
\usepackage{mathdots}

\usepackage{float}
\usepackage{caption}
\usepackage{setspace}
\usepackage[letterpaper, left=0.9in, right=0.9in, top=0.8in, bottom=1in]{geometry}

\usepackage[bottom]{footmisc}
\usepackage[colorlinks,linkcolor={blue!75!black},urlcolor=red,citecolor=ForestGreen,backref]{hyperref}
\usepackage{footnotebackref}
\renewcommand*\backref[1]{\ifx#1\relax \else (Cited on #1) \fi}

\usepackage{threeparttable}

\usepackage{arydshln}

\usepackage{enumerate}
\usepackage{enumitem} 

\usepackage{lipsum}

\usepackage{natbib}
\usepackage{amsmath}
\allowdisplaybreaks[4]

\title{\bf Generalization Bounds of Nonconvex-(Strongly)-Concave Stochastic Minimax Optimization}

\author{
Siqi Zhang \footnote{Equal contribution.} \thanks{Department of Applied Mathematics and Statistics, Johns Hopkins University, USA. \texttt{szhan207@jhu.edu}}
\and
Yifan Hu \footnotemark[1]\ \thanks{Risk Analytics and Optimization Chair, EPFL, Switzerland. \texttt{yifan.hu@epfl.ch}}
\and
Liang Zhang \thanks{Department of Computer Science, ETH Z\"urich, Switzerland. \texttt{liang.zhang@inf.ethz.ch}, \texttt{niao.he@inf.ethz.ch}}
\and
Niao He \footnotemark[4]
}

\begin{document}

\maketitle

\begin{abstract}
    This paper takes an initial step to systematically investigate the generalization bounds of algorithms for solving nonconvex-(strongly)-concave (NC-SC\,/\,NC-C) stochastic minimax optimization measured by the stationarity of primal functions. We first establish \textit{algorithm-agnostic generalization bounds} via \emph{uniform convergence} between the empirical minimax problem and the population minimax problem. The sample complexities for achieving $\eps$-generalization are $\tilde{\mathcal{O}}\autopar{d\kappa^2\epsilon^{-2}}$ and $\tilde{\mathcal{O}}\autopar{d\epsilon^{-4}}$ for NC-SC and NC-C settings, respectively, where $d$ is the dimension and $\kappa$ is the condition number. We further study the \textit{algorithm-dependent generalization bounds} via stability arguments of algorithms. In particular, we introduce a novel stability notion for minimax problems and build a connection between generalization bounds and the stability notion. As a result, we establish \textit{algorithm-dependent generalization bounds} for \emph{stochastic gradient descent ascent (SGDA)} algorithm and the more general \emph{sampling-determined algorithms}.
\end{abstract}

\section{Introduction}
\label{sec:intro}

In this paper, we consider stochastic minimax problems:
\begin{equation}
    \label{eq:main_problem}
    \min_{x\in \mathcal{X}}\max_{y\in\mathcal{Y}}\ F(x,y)\triangleq \E_\xi\automedpar{f\autopar{x,y;\xi}},
\end{equation}
where $\mathcal{X}\subseteq\mathbb{R}^{d}$ and $\mathcal{Y}\subseteq\mathbb{R}^{d'}$ ($d, d'\in\mathbb{N}_+$) are two nonempty closed convex sets, $\xi\in\Xi$ is a random variable following an unknown distribution $\mathcal{D}$, and $f:\mathcal{X}\times\mathcal{Y}\times\Xi\rightarrow\mathbb{R}$ is continuously differentiable and Lipschitz smooth jointly in $x$ and $y$ for any $\xi$. We denote the objective \eqref{eq:main_problem} as the \emph{population minimax problem}. Throughout the paper, we focus on the case where $F$ is nonconvex in $x$ and (strongly)-concave in $y$, i.e., \emph{nonconvex-(strongly)-concave (NC-SC~/~NC-C)}. Such minimax problems appear ubiquitously in practical applications, including adversarial training \citep{madry2018towards,wang2019convergence}, generative adversarial networks (GANs) \citep{goodfellow2014generative,sanjabi2018convergence,lei2020sgd}, reinforcement learning \citep{dai2017learning,dai2018sbeed,huang2020convergence} and robust training \citep{sinha2018certifying}.

Although the distribution $\mathcal{D}$ often remains unknown, one generally has access to a dataset $S=\autobigpar{\xi_1,\cdots,\xi_n}$ consisting of $n$ independently and identical distributed (i.i.d.) samples from $\mathcal{D}$. Correspondingly, researchers resort to solving an \emph{empirical minimax problem}:
\begin{equation}
    \label{eq:FS_problem}
    \min_{x\in \mathcal{X}}\max_{y\in\mathcal{Y}}\ F_S(x,y)\triangleq \frac{1}{n}\sum_{i=1}^nf(x,y;\xi_i).
\end{equation} 
A natural question arises: 
\begin{center}
    \emph{How does the output of an algorithm $\mathcal{A}$ for solving the empirical minimax problem\\  generalizes on the population minimax problem?}
\end{center}

We first specify the measurement. Since functions $F$ and $F_S$ are nonconvex in $x$, finding their global optimal solutions is generally intractable. Instead, one aims to design an algorithm $\mathcal{A}$ that finds an
\emph{$\eps$-stationary point} \cite{lin2020gradient},
i.e.,
\begin{equation*}
    \autonorm{\nabla\Phi(\mathcal{A}_x(S))}\leq\epsilon
    \quad \text{or} \quad
    \dist\autopar{0,\partial\Phi(\mathcal{A}_x(S))}\leq\epsilon,
\end{equation*}
where $\Phi(x)\triangleq\max_{y\in\mathcal{Y}}F(x,y)$ and $\Phi_S(x)\triangleq\max_{y\in\mathcal{Y}}F_S(x,y)$ are primal functions, $\mathcal{A}_x(S)$ is the $x$-component of the output of any algorithm $\mathcal{A}$ for solving \eqref{eq:FS_problem}, $\dist\autopar{y,X}\triangleq\inf_{x\in X}\autonorm{y-x}$ and $\partial\Phi$ is the (Fr\'echet) subdifferential of $\Phi$. When $\Phi$ is nonsmooth, we resort to the gradient norm of its Moreau envelope 
to measure the first-order stationarity as it provides an upper bound on $\dist\autopar{0,\partial\Phi(\cdot)}$ \citep{davis2019stochastic}.

Taking the gradient norm as an example, the error for solving the population minimax problem \eqref{eq:main_problem} via solving its empirical counterpart \eqref{eq:FS_problem} consists of two terms\footnote{For the simplicity of demonstration, here we assume there is no constraint, and primal functions are differentiable. We will formally introduce the detailed settings in Section \ref{sec:prelim}.}:
\begin{equation}
\label{eq:opt_decomposition}
    \begin{split}
        \EE\autonorm{\nabla \Phi(\mathcal{A}_x(S))}
        \leq
        \underbrace{\EE\autonorm{\nabla \Phi_S(\mathcal{A}_x(S))}}_{\text{optimization error}}+\underbrace{\EE\autonorm{\nabla \Phi(\mathcal{A}_x(S))-\nabla \Phi_S(\mathcal{A}_x(S))}}_{\text{generalization error}}.
    \end{split}
\end{equation}
Such decomposition on the gradient norm also appears in nonconvex minimization, e.g., \cite{foster2018uniform,mei2018landscape,davis2022graphical,lei2022stability}. 
The optimization error corresponds to the error of solving the empirical minimax problem \eqref{eq:FS_problem} and has been widely studied ~\citep{luo2020stochastic,yang2020catalyst}. On the other hand, the generalization error for minimax problems remains largely unexplored. A few recent works~\citep{farnia2021train,zhang2021generalization,lei2021stability,ozdaglar2022good}
studied the generalization performances in minimax optimization measured by the function value-based gap. However, these do not fit well in the nonconvex setting since the optimization part remains intractable.

The goal of our paper is to characterize the generalization error 
$$\EE\autonorm{\nabla \Phi(\mathcal{A}_x(S))-\nabla \Phi_S(\mathcal{A}_x(S))}.$$ 
It is not easy as both $\Phi_S(\cdot)$ and $\mathcal{A}_x(S)$ depend on the dataset $S$, which induces correlation issues when taking expectation. To address such dependence issue, one may use \emph{uniform convergence} or \emph{stability arguments}.

Uniform convergence characterizes the difference between the empirical minimax optimization and the population minimax problem on worst $x\in\mathcal{X}$, i.e., 
$$\EE \sup_{x\in\mathcal{X}}\autonorm{\nabla\Phi(x)-\nabla\Phi_S(x)}.$$
Although uniform convergence has been extensively studied in stochastic optimization~\citep{kleywegt2002sample,mei2018landscape,davis2022graphical}, a key difference for stochastic minimax optimization is that the primal function is the average over $n$ i.i.d. random functions. Thus existing techniques in uniform convergence for classical stochastic optimization are not directly applicable here. One needs to additionally characterize the differences between maximizers 
$$\argmax_{y\in\mathcal{Y}}F_S(x,y)\quad\text{and}\quad\argmax_{y\in\mathcal{Y}}F(x,y).$$
Note that uniform convergence is invariant to the choice of algorithms and provides an upper bound on the generalization error for any $\mathcal{A}_x(S)\in\mathcal{X}$. Thus the derived generalization bound is \emph{algorithm-agnostic} that applies to any algorithms. Since it is the worst case over all $x\in\mathcal{X}$, the derived bounds generally involve the dimension of $x$.

We further investigate generalization bounds via stability arguments. This approach analyzes the stability of specific algorithms and builds a connection between stability and generalization. It has been extensively studied in stochastic minimization \citep{bousquet2002stability,shalev2010learnability, hardt2016train,klochkov2021stability} and minimax problems recently~\citep{farnia2021train,lei2021stability,boob2021optimal,yang2022differentially}.  Most of these work focuses on the measurement of the function-value gap. Under such measurement, generalization follows directly from stability. However, for the measurement of stationarity for nonconvex problems, building up a link between stability and generalization becomes significantly more challenging.  Compared to uniform convergence, the stability-based generalization bound is generally independent of the dimension $d$. As it requires a case-by-case analysis of stability for different algorithms,  it is \textit{algorithm-dependent}. We particularly study the generalization of the widely used \emph{stochastic gradient descent ascent (SGDA)}~\citep{farnia2021train} and a class of algorithms that extends SGDA called \emph{sampling-determined algorithms} (see Definition \ref{defn:sampling-determined_alg})~\citep{lei2022stability}.

\subsection{Contributions}
In this paper, we initiate a systematic study on the generalization bounds (see Table \ref{table:summary_results}) for nonconvex stochastic minimax problems from both \textit{uniform convergence} and \textit{stability argument} perspectives.  Our contributions are two-fold:

\begin{itemize}[leftmargin = 2em]
    \item We establish the first uniform convergence results between the population and the empirical nonconvex minimax optimization in NC-SC and NC-C settings, measured by stationarity. Our results provide an \textit{algorithm-agnostic} generalization bound for any algorithms that solve empirical nonconvex minimax problems. Specifically, the sample complexities to achieve an $\eps$-uniform convergence or an $\eps$-generalization error are $\tilde{\mathcal{O}}\autopar{d\kappa^2\epsilon^{-2}}$ and $\tilde{\mathcal{O}}\autopar{d\epsilon^{-4}}$ for the NC-SC and NC-C settings, respectively.
    \item From the stability argument perspective, we first introduce a novel stability measurement based on stationarity; then, we establish the connection between the stability and the generalization error of an algorithm in both NC-SC and NC-C settings. We further provide the \textit{algorithm-dependent} generalization error measured by the stationarity for the classical SGDA algorithm and sampling-determined algorithms utilizing their stability.
\end{itemize}

\begin{table}[tbp]
    \centering
    \small
    \renewcommand{\arraystretch}{2}
    \begin{threeparttable}[b]
        \begin{tabular}{c |
        >{\centering}p{0.2\textwidth}|
        >{\centering}p{0.25\textwidth}|
        >{\centering\arraybackslash}p{0.25\textwidth}}
            \hline \hline
            \multirow{2}{*}{\diagbox{\textbf{Setting}\tnote{1}}{\textbf{Approach}}}
            & \multirow{2}{*}{\textbf{Uniform Convergence}}
            & \multicolumn{2}{c}{\textbf{Stability Argument}}
            \\
            \cline{3-4}
            & 
            & \textbf{SGDA}
            & \textbf{Sampling-determined Alg.}
            \\
            \hline \hline
            \textbf{NC-SC}
            & \makecell[c]{$\Tilde{\mathcal{O}}\autopar{\kappa\sqrt{\frac{d}{n}}}$\\
            Theorem \ref{thm:NC-SC Agnostic}
            }
            & \makecell[c]{$\mathcal{O}\autopar{\kappa^{1+\zeta_1}\autopar{\frac{T^{1-\zeta_1}}{n}+\frac{1}{\sqrt{n}}}}$
            \\
            Corollary \ref{thm:NC-SC SGDA Specific}
            }
            & \makecell[c]{$\mathcal{O}\autopar{\kappa \autopar{\sqrt{\frac{T}{n}}+\frac{1}{\sqrt{n}}}}$\\
            Corollary \ref{thm:NC-SC SDA}
            }
            \\
            \hline
            
            \textbf{NC-C}
            & \makecell[c]{$\Tilde{\mathcal{O}}\autopar{\autopar{\frac{d}{n}}^{1/4}}$
            \\
            Theorem \ref{thm:NC-C Agnostic}
            }
            & \makecell[c]{$\mathcal{O}\autopar{\autopar{\frac{T^{1-\zeta_2}}{n}}^{1/6} + \autopar{\frac{1}{n}}^{1/8}}$
            \\
            Corollary \ref{thm:NC-C SGDA Specific}
            }
            & \makecell[c]{$\mathcal{O}\autopar{\autopar{\frac{T}{n}}^{1/12} + \autopar{\frac{ 1}{n}}^{1/8}}$
            \\
            Corollary \ref{thm:NC-C SDA}
            }
            \\
            \hline \hline
        \end{tabular}
        \begin{tablenotes}
            \item[1] $\tilde{O}(\cdot)$ hides logarithmic factors, $d$: the dimension of $\mathcal{X}$, $n$: sample size, $\kappa$: condition number $\frac{L}{\mu}$ \\
            $L$: Lipschitz smoothness parameter, $\mu$: strong concavity parameter, $T$: iteration number of algorithms\\
            $\zeta_{1}, \zeta_{2} \in(0,1)$: constants depending on stepsizes, refer to Corollary \ref{thm:NC-SC SGDA Specific} and \ref{thm:NC-C SGDA Specific} for details. SGDA has specific requirements on stepsize, while sampling-determined algorithms do not have restrictions on stepsize.
 		\end{tablenotes}
    \end{threeparttable}
    \caption{Summary of Generalization Bounds for Nonconvex Stochastic Minimax Optimization}
    \label{table:summary_results}
\end{table}

\subsection{Literature Review}

\paragraph{Nonconvex Minimax Optimization}
Various algorithms have been proposed to solve NC-SC minimax optimization~\citep{nouiehed2019solving, lin2020gradient,lin2020near,luo2020stochastic,yang2020global,boct2020alternating,xu2020unified,lu2020hybrid,yan2020optimal,guo2021novel,sharma2022federated,zhang2022sapd+}. 
For stochastic NC-SC minimax problems, \citet{zhang2022sapd+} achieves the best-known complexity of $\mathcal{O}(\kappa\epsilon^{-4})$, and $\mathcal{O}(\kappa^2\epsilon^{-3})$ result with additional individual smoothness assumption. Also, \citet{yang2022faster} introduced a stochastic smoothed-AGDA algorithm which achieves the best-known $\mathcal{O}(\kappa^2\epsilon^{-4})$ for single-loop algorithms. The lower bounds of NC-SC problems are extensively studied in several recent works \citep{zhang2021complexity, han2021lower, li2021complexity}.

The primal function for NC-SC problems is generally smooth, while the primal functions can be nonsmooth for NC-C problems~\citep{thekumparampil2019efficient,lin2020gradient}.
Recent years witnessed a surge of algorithms for NC-C problems in deterministic, finite-sum, and stochastic settings, e.g., \citep{zhang2020single, ostrovskii2021efficient, thekumparampil2019efficient, zhao2020primal,nouiehed2019solving,yang2020catalyst,lin2020gradient,boct2020alternating,rafique2021weakly}, to name a few. To the best of our knowledge, \cite{thekumparampil2019efficient,yang2020catalyst,lin2020near} achieved the best known $\tilde\cO(\eps^{-3})$ complexity in the deterministic case, while \cite{yang2020catalyst} achieved the best known $\tilde\cO(n^{3/4}\eps^{-3})$ complexity in the finite-sum case, and \cite{zhang2022sapd+} provided the best known $\cO(\eps^{-6})$ complexity in the stochastic case. These works differ from our paper in that we aim to characterize the generalization error of algorithms while they focus mainly on the optimization error of the algorithms.

\paragraph{Uniform Convergence}
A series of works from stochastic optimization and statistical learning theory studied uniform convergence on the worst-case differences between the population objective $L(x)$ and its empirical objective $L_S(x)$ constructed via sample average approximation (SAA, also known as empirical risk minimization). Interested readers may refer to prominent results in statistical learning~\citep{fisher1922mathematical, vapnik1999overview, van2000asymptotic}. 
For finite-dimensional problem,  \citet{kleywegt2002sample} showed that the sample complexity is $\cO(d\eps^{-2})$ to achieve an $\eps$-uniform convergence in high probability, i.e., $\mathbb{P}(\sup_{x\in\mathcal{X}} |L(x) - L_S(x)|\geq \eps)$. 
For nonconvex empirical objectives, 
\citet{mei2018landscape} and \citet{davis2022graphical} established $\tilde \cO(d\eps^{-2})$ sample complexity of uniform convergence measured by 
the stationarity for nonconvex smooth and weakly convex functions, respectively. 
In addition, \citet{wang2017differentially} used uniform convergence to demonstrate the generalization and the gradient complexity of differentially private algorithms for stochastic optimization. Recently, \citet{amir2022thinking} demonstrated the generalization error of gradient descent on a generalized linear model using uniform convergence and showed that the stability argument is insufficient to achieve generalization.   
To the best of our knowledge, our paper is the first to study uniform convergence for nonconvex minimax optimization.

\paragraph{Stability-Based Generalization Bounds}
This line of research focuses on analyzing generalization bounds of stochastic optimization via the uniform stability property of specific algorithms, including SAA~\citep{bousquet2002stability, shalev2009stochastic}, stochastic gradient descent~\citep{hardt2016train,bassily2020stability,lei2022stability}, and uniformly stable algorithms \citep{klochkov2021stability}. Recently, a series of works further studied the generalization performances measured by the function-value gap of various algorithms in minimax problems. \citet{farnia2021train} gave the generalization bound for the outputs of gradient-descent-ascent (GDA) and proximal-point algorithm (PPA) in both (strongly)-convex-(strongly)-concave and nonconvex-nonconcave smooth minimax problems. \citet{lei2021stability} focused on GDA and provided a comprehensive study for different settings of minimax problems with generalization measured by function-value gaps. \citet{boob2021optimal} provided stability and generalization results of extragradient algorithm (EG) in the smooth convex-concave setting. On the other hand, \citet{zhang2021generalization} studied stability and generalization of the empirical minimax problem under the (strongly)-convex-(strongly)-concave setting, assuming that one can find the optimal solution of the empirical minimax problem. Our work differs from those in that we propose a novel stability notion for minimax optimization measured by stationarity and build up the link between such stability and generalization error.

\section{Problem Setting}
\label{sec:prelim}

\paragraph{Notations}
Throughout the paper, we use $\autonorm{\cdot}$ as the $\ell_2$-norm, $\nabla f=\autopar{\nabla_x f, \nabla_y f}$ as the gradient of a function $f$, for nonnegative functions $f$ and $g$, we say $f=\mathcal{O}\autopar{g}$ if $f(x)\leq cg(x)$ for some $c>0$. We denote $\proj_\mathcal{X}(x')\triangleq\argmin_{x\in\mathcal{X}}\autonorm{x-x'}^2$ as the projection operator. 
Let $\mathcal{A}(S)\triangleq\autopar{\mathcal{A}_x(S),\mathcal{A}_y(S)}$ denote the output of an algorithm $\mathcal{A}$ on the empirical minimax problem \eqref{eq:FS_problem} with dataset $S$.
Given $\mu\geq 0$, we say a function $g:\mathcal{X}\rightarrow\mathbb{R}$  is
$\mu$-strongly convex if $ g(x)-\frac{\mu}{2}\autonorm{x}^2 $ is convex, and it is $\mu$-strongly concave if $-g$ is $\mu$-strongly convex. Function $g$ is $\mu$-weakly convex if $g(x)+\frac{\mu}{2}\autonorm{x}^2 $ is convex
(see more notations and standard definitions in Appendix \ref{apdx_sec:notations}).

\begin{definition}[Smooth Function] 
    \label{defn:L_smooth_minimax}
We say a function $f:\mathcal{X}\times\mathcal{Y}\rightarrow\mathbb{R}$ is $L$-smooth jointly in $(x,y)$ if the function is continuously differentiable, and there exists a constant $L>0$ such that for any $(x_1,y_1), (x_2,y_2)\in\mathcal{X}\times\mathcal{Y}$, we have 
    \begin{equation*}
        \begin{split}
            \autonorm{\nabla_x f(x_1,y_1)-\nabla_x f(x_2,y_2)}&\leq L(\|x_1-x_2\|+\|y_1-y_2\|),\\
            \autonorm{\nabla_y f(x_1,y_1)-\nabla_y f(x_2,y_2)}&\leq L(\|x_1-x_2\|+\|y_1-y_2\|).
        \end{split}
    \end{equation*}
\end{definition}

By definition, it is easy to find that an $L$-smooth function is also $L$-weakly convex. Next, we introduce the main assumptions used throughout the paper.

\begin{assumption}[Main Settings]
    \label{as:ncsc-ncc}
    We assume the following:
    \begin{itemize}[leftmargin = 2em]
        \item The function $f(x,y;\xi)$ is $L$-smooth jointly in $(x,y)\in\mathcal{X}\times \mathcal{Y}$ for any $\xi$.
        \item The function $f(x,y;\xi)$ is $\mu$-strongly concave in $y\in\mathcal{Y}$ for any $x\in\mathcal{X}$ and any $\xi$ where $\mu\geq 0$.
        \item The gradient norms of $f(\cdot,\cdot;\xi)$ are bounded by $G$ respectively for any $\xi$. 
        \item The domains $\mathcal{X}$ and $\mathcal{Y}$ are compact convex sets, i.e., there exists constants $D_\mathcal{X}, D_\mathcal{Y}>0$ such that  for any $x\in\mathcal{X}$, $\|x\|^2\leq D_\mathcal{X}$ and for any $y\in\mathcal{Y}$, $\|y\|^2\leq D_\mathcal{Y}$, respectively.
    \end{itemize}
\end{assumption}
Note that compact domain assumption appears widely in uniform convergence literature~\citep{kleywegt2002sample,davis2022graphical}. 
Under Assumption \ref{as:ncsc-ncc}, the objective function $F$ is $L$-smooth in $(x,y)$ and $\mu$-strongly concave for any $\xi$. When $\mu>0$, we call the population minimax problem \eqref{eq:main_problem} a \emph{nonconvex-strongly-concave} (NC-SC) minimax problem; when $\mu=0$, we call it a \emph{nonconvex-concave} (NC-C) minimax problem. 

\begin{definition}[Moreau Envelope] 
For an $L$-weakly convex function $\Phi$ and $0<\lambda<1/L$, we use $\Phi^\lambda(x)$ and $\prox_{\lambda \Phi}(x)$ to denote the the Moreau envelope of $\Phi$ and the proximal point of $\Phi$ for a given point $x$, defined as following:
\begin{equation}
    \begin{split}
        \Phi^\lambda(x)\triangleq \min_{z\in\mathcal{X}}\autobigpar{\Phi(z)+\frac{1}{2\lambda}\autonorm{z-x}^2},
        \quad
        \prox_{\lambda \Phi}(x)\triangleq \argmin_{z\in\mathcal{X}}\autobigpar{\Phi(z)+\frac{1}{2\lambda}\autonorm{z-x}^2}.
    \end{split}
\end{equation}
\end{definition}

Below we recall some important properties on the primal function $\Phi$ and its Moreau envelope $\Phi^{\lambda}(x)$ presented in the literature \citep{davis2019stochastic,thekumparampil2019efficient,lin2020gradient,lin2020near}.
\begin{lemma}[Properties of $\Phi$ and $\Phi^\lambda$]
    \label{lm:properties_Phi_envelope}
    In the NC-SC setting ($\mu>0$), both $\Phi(x)$ and $\Phi_S(x)$ are $\tilde{L}\triangleq L(1+\kappa)$-smooth with the condition number $\kappa\triangleq L/\mu$, 
    In the NC-C setting ($\mu=0$), the primal function $\Phi$ is $L$-weakly convex, its Moreau envelope $\Phi^{\lambda}(x)$ is Lipschitz smooth, also $ \nabla\Phi^{\lambda}(x)=\lambda^{-1}(x-\hat{x})$, $\autonorm{\nabla\Phi^{\lambda}(x)}\geq\dist\autopar{0,\partial\Phi(\hat{x})}$, where $\hat{x}=\prox_{\lambda\Phi}(x)$ and $0<\lambda<1/L$.
\end{lemma}

\paragraph{Performance Measurement}  
In the NC-SC setting, the primal functions $\Phi$ and $\Phi_S$ are both $\tilde{L}$-smooth. Regarding the constraint, we measure the difference between the population and empirical minimax problems using \emph{the generalized gradient of the population and the empirical primal functions}, i.e., $\EE\autonorm{\mathcal{G}_\Phi(\mathcal{A}_x(S))-\mathcal{G}_{\Phi_S}(\mathcal{A}_x(S))}$, where $\mathcal{G}_\Phi(x)\triangleq \tilde{L}(x-\proj_\mathcal{X}(x-(1/\tilde{L})\nabla\Phi(x)))$. The following inequality summarized the relationship of measurements in terms of generalized gradient and in terms of gradient used in Section \ref{sec:intro}.
\begin{equation}
\label{eq:relationship_1}
    \begin{split}
    \underbrace{\EE\autonorm{\mathcal{G}_\Phi(\mathcal{A}_x(S))-\mathcal{G}_{\Phi_S}(\mathcal{A}_x(S))}}_{\text{generalization error of Algorithm $\mathcal{A}$}}
\leq 
    \EE\autonorm{\nabla\Phi(\mathcal{A}_x(S))-\nabla\Phi_S(\mathcal{A}_x(S))}
\leq
    \underbrace{\EE \Big[\max_{x\in\mathcal{X}} \autonorm{\nabla\Phi(x)-\nabla\Phi_S(x)}\Big],}_{\text{algorithm-agnostic uniform convergence}}
    \end{split}
\end{equation}
where the first inequality holds as projection is a non-expansive operator. The term in the left-hand side (LHS) above is the generalization error of an algorithm $\mathcal{A}$ we desire in the NC-SC case.

For the NC-C case, the primal function $\Phi(x)$ is $L$-weakly convex, we use the gradient of its Moreau Envelope to characterize the (near)-stationarity \citep{davis2019stochastic}. We measure the difference between the population and empirical problems using the difference between \emph{the gradients of their respective Moreau envelopes}. 

The generalization error and the uniform convergence in the  NC-C case should be given as follows:
\begin{equation}
\label{eq:relationship_2}
    \begin{split}
    \underbrace{\EE \Big\|\nabla\Phi^{1/(2L)}(\mathcal{A}_x(S))-\nabla\Phi_S^{1/(2L)}(\mathcal{A}_x(S))\Big\|}_{\text{generalization error of Algorithm $\mathcal{A}$}}
\leq
    \underbrace{\EE \Big[\max_{x\in\mathcal{X}} \Big\|\nabla\Phi^{1/(2L)}(x)-\nabla\Phi_S^{1/(2L)}(x)\Big\|\Big]}_{\text{algorithm-agnostic uniform convergence}}.
    \end{split}
\end{equation}
The term in the LHS above is the generalization error of an algorithm $\mathcal{A}$ we desire in the NC-C case.

\section{Uniform Convergence and Generalization}
\label{sec:main_result}
In this section, we discuss the sample complexity for achieving $\eps$-uniform convergence and $\eps$-generalization error for NC-SC and NC-C  stochastic minimax optimization.

\subsection{NC-SC Stochastic Minimax Optimization}
Under the NC-SC setting, the next theorem demonstrates the uniform convergence between gradients of primal functions of the population and empirical minimax problem, which is a structural property of the empirical and population minimax problem. We defer the proof to Appendix \ref{apdx_sec:proof_NCSC_unif_convergence}.

\begin{theorem}[Uniform Convergence, NC-SC]
    \label{thm:NC-SC Agnostic}
    Under Assumption \ref{as:ncsc-ncc} with $\mu>0$, we have
    \begin{equation}
        \EE \automedpar{\max_{x\in\mathcal{X}}\autonorm{\nabla\Phi(x)-\nabla \Phi_S(x)}}= \Tilde{\mathcal{O}}\autopar{d^{1/2}\kappa n^{-1/2}}.
    \end{equation}
    Furthermore, to achieve $\eps$-uniform convergence and $\eps$-generalization error for any algorithm $\mathcal{A}$ such that  $\EE\autonorm{\mathcal{G}_\Phi(\mathcal{A}_x(S))-\mathcal{G}_{\Phi_S}(\mathcal{A}_x(S))}\leq\eps$, it suffices to have
    \begin{equation}
        n= n_{\mathrm{NCSC}}^* \triangleq \Tilde{\cO}\autopar{d\kappa^2\eps^{-2}}.
    \end{equation}
\end{theorem}
To the best of our knowledge, it is the first uniform convergence and algorithm-agnostic generalization error bound result for NC-SC stochastic minimax problem. In comparison, existing works on generalization error analysis of minimax problems \citep{farnia2021train,lei2021stability} using stability arguments are algorithm-specific and can only handle function-value gap measurement. \citet{zhang2021generalization} establish algorithm-agnostic stability and generalization in the strongly-convex-strongly-concave regime, yet their analysis does not extend to the nonconvex regime. 

Error decomposition \eqref{eq:opt_decomposition} and Theorem \ref{thm:NC-SC Agnostic} imply that for any algorithm that achieves an $\eps$-stationarity point of the empirical minimax problem, its sample complexity for finding an $\eps$-stationary point of the population minimax problem is at most $\Tilde{\cO}\autopar{d\kappa^2\eps^{-2}}$. 
Such an observation provides generalization guarantees for any algorithms that solve finite-sum (empirical) minimax problems. It is especially useful for some SOTA algorithms like Catalyst-SVRG \citep{zhang2021complexity} and finite-sum version SREDA \citep{luo2020stochastic} that are complicated and thus there lack stability analysis and generalization bounds for them. 

\paragraph{Proof Sketch}
We briefly discuss the proof of Theorem \ref{thm:NC-SC Agnostic}.

\noindent\textbf{Step 1}:
First, we use a $\upsilon$-net $\{x_k\}_{k=1}^Q$~\citep{vapnik1999overview} to decompose the error and handle the dependence issue between $\argmax_{x\in\mathcal{X}}\autonorm{\nabla \Phi_S(x)- \nabla\Phi(x)}$ and $\Phi_S(x)$.

\noindent\textbf{Step 2}:
For any $x_k$ within the $\upsilon$-net, we have the error following decomposition
$$
\begin{aligned}
\|\nabla \Phi_S(x_k)- \nabla\Phi(x_k)\|
\leq 
(\|\nabla \Phi_S(x_k)- \nabla\Phi(x_k)\|-\EE\|\nabla \Phi_S(x)- \nabla\Phi(x_k)\|)
\qquad+ \EE\|\nabla \Phi_S(x_k)- \nabla\Phi(x_k)\|.
\end{aligned}
$$
When bounding  $\EE\|\nabla \Phi_S(x_k)- \nabla\Phi(x_k)\|$ in the right-hand side (RHS), we need to characterize the difference between $\argmax_{y\in\mathcal{Y}}F_S(x,y)$ and $\argmax_{y\in\mathcal{Y}}F(x,y)$ using the stability argument of sample average approximation~\citep{shalev2009stochastic}. This step appears uniquely for minimax optimization due to the special structure of the primal function $\Phi_S(x)=\max_y \frac{1}{n}\sum_{i=1}^n f(x,y;\xi_i)$, which is not the average over $n$ random functions. Then we utilize the established stability argument to show that the first term in the RHS is sub-Gaussian and apply the concentration inequality, which leads to the result. 
\qed

\subsection{NC-C Stochastic Minimax Optimization}
\label{sec:uniform_NCC}

In this subsection, we derive the uniform convergence and algorithm-agnostic generalization bounds for NC-C stochastic minimax problems. Recall that the primal function $\Phi$ is  $L$-weakly convex \citep{thekumparampil2019efficient}, and thus that $\nabla \Phi$ is not well-defined. We use the gradient norm of the Moreau envelope of the primal function as the measurement \citep{davis2019stochastic}.

\begin{theorem}[Uniform Convergence, NC-C]
    \label{thm:NC-C Agnostic}
    Under Assumption \ref{as:ncsc-ncc} with $\mu=0$, we have
    \begin{equation}
        \EE\automedpar{\max_{x\in\mathcal{X}}\autonorm{\nabla  \Phi_S^{1/(2L)}(x)-\nabla \Phi^{1/(2L)}(x)}}= \Tilde{\mathcal{O}}\autopar{d^{1/4}n^{-1/4}}.
    \end{equation}
    Furthermore, to achieve $\eps$-uniform convergence and $\eps$-generalization error for any algorithm $\mathcal{A}$ such that  $\EE\automedpar{\autonorm{\nabla\Phi^{1/(2L)}(\mathcal{A}_x(S))-\nabla\Phi_S^{1/(2L)}(\mathcal{A}_x(S))}}\leq\eps$, it suffices to have
    \begin{equation}
        n= n_{\mathrm{NCC}}^* \triangleq \tilde\cO\autopar{d\eps^{-4}}.
    \end{equation}
\end{theorem}

We defer the detailed proof to Appendix \ref{apdx_sec:proof_NCC_unif_convergence}. To the best of our knowledge, this is the first algorithm-agnostic generalization error result in NC-C stochastic minimax optimization. Similar to the NC-SC setting,  Theorem \ref{thm:NC-C Agnostic} with error decomposition \eqref{eq:opt_decomposition} provides generalization guarantees for any algorithms that solve the NC-C empirical minimax problem, including the best-known Catalyst algorithm~\citep{yang2020catalyst}. More specifically, if an algorithm finds an $\eps$-stationarity point of the empirical minimax problem, with sample size $n = \Tilde{\cO}\autopar{d\eps^{-4}}$, the point is also an $\cO(\eps)$-stationarity point of the population minimax problem.

\paragraph{Proof Sketch}

The analysis of Theorem \ref{thm:NC-C Agnostic} builds up a link between NC-C and NC-SC settings and consists of three parts.  

\noindent\textbf{Step 1}:
By definition of the gradient of the Moreau envelope, it holds that 
$$
\|\nabla  \Phi_S^{\lambda}(x)-\nabla \Phi^{\lambda}(x)\| \leq \frac{1}{\lambda}\|\prox_{\lambda \Phi}(x) - \prox_{\lambda \Phi_S}(x)\|.
$$
We first use a  $\upsilon$-net $\{x_k\}_{k=1}^Q$~\citep{vapnik1999overview} to handle the dependence issue between $\tilde x^*\in \argmax_{x\in\mathcal{X}}\|\prox_{\lambda \Phi}(x) - \prox_{\lambda \Phi_S}(x)\|$ and $\Phi_S$.

\noindent\textbf{Step 2}:
We introduce the following $\ell_2$-regularized minimax problem:
\begin{equation*}
    \min_{x\in \mathcal{X}}\max_{y\in\mathcal{Y}}\ F(x,y)-\frac{\nu}{2}\|y\|^2.
\end{equation*}
Notice that this problem is NC-SC. We further build a connection between NC-C stochastic minimax optimization problems and the corresponding regularized NC-SC stochastic minimax optimization problems. Then we carefully choose the regularization parameter $\nu$ to derive the uniform convergence. 

The following lemma characterizes the distance between the proximal points of the primal function from the original NC-C problem and its regularized NC-SC problem. Note that the lemma may be of independent interest for the design and the analysis of gradient-based methods for NC-C problems.  
\begin{lemma}
\label{lemma:nc_sc_to_ncc}
For $\nu>0$, denote 
$
    \hat \Phi(x) = \max_{y\in\mathcal{Y}} F(x,y)-\frac{\nu}{2}\|y\|^2
$ as the primal function of the regularized NC-C problem. It holds for $\lambda\in(0, (L+\nu)^{-1})$ that 
$$
\|\prox_{\lambda \Phi}(x) - \prox_{\lambda \hat \Phi}(x)\|^2 \leq \frac{\nu D_\mathcal{Y}\lambda}{1-\lambda(L+\nu)}.
$$
\end{lemma}
The above lemma implies that for small regularization parameter $\nu$, the difference between the proximal point of the primal function $\Phi$ of the NC-C problem and the primal function $\hat \Phi$ of the regularized NC-SC problem is small.

\noindent\textbf{Step 3}:
It remains to characterize the distance between $\prox_{\lambda \hat \Phi}(x)$ and $\prox_{\lambda \hat \Phi_S}(x)$, where $\hat \Phi_S$ is the primal function of the regularized empirical minimax problem. By definition of $\prox_{\lambda \hat \Phi}(x)$ and $\prox_{\lambda \hat \Phi_S}(x)$, the distance is equivalent to the difference between the optimal solutions on $x$ of a strongly-convex strongly-concave (SC-SC) population minimax problem and the counterpart empirical minimax problem. We utilize the existing stability-based results for SC-SC minimax optimization~\cite {zhang2021generalization} to the upper bound such a distance. We further show that $\|\prox_{\lambda \hat \Phi_S}(x) - \prox_{\lambda \hat \Phi}(x) \|-\EE\|\prox_{\lambda \hat \Phi_S}(x) - \prox_{\lambda \hat \Phi}(x) \|$ is a sub-Gaussian random variable and apply concentration inequality.\qed

\subsection{Comparing Uniform Convergence for Minimization, NC-SC, and NC-C Minimax Problems}
For general stochastic nonconvex optimization $\min_{x\in\mathcal{X}} \EE [f(x;\xi)]$, the sample complexity of achieving $\eps$-uniform convergence,
$$
\mathbb{E}\max_{x\in\mathcal{X}} \Big\| \frac{1}{n}\sum_{i=1}^n \nabla f(x;\xi_i) - \mathbb{E} \nabla f(x;\xi)\Big\|\leq \eps,
$$
between the gradient of the population problem and the empirical problem is $\tilde \cO(d\eps^{-2})$~\citep{davis2022graphical,mei2018landscape}. For nonconvex minimax optimization, if we care about the uniform convergence in terms of the gradient of $F$, i.e., 
$$
\EE\max_{x\in\mathcal{X},y\in\mathcal{Y}} \autonorm{ \frac{1}{n}\sum_{i=1}^n \nabla f(x,y;\xi_i) - \mathbb{E} \nabla f(x,y;\xi)},
$$
where $\nabla f$ denotes the full gradient with respect to $x$ and $y$, 
existing analysis in \citet{mei2018landscape} directly gives a $\tilde\cO(d\eps^{-2})$ sample complexity.
However, since we care about the gradient of the primal function, the analysis becomes more complicated, which we detail in the following. 
\begin{enumerate}[leftmargin = 2em]
\item  
    In the NC-SC setting, to establish uniform convergence, we bound
    \begin{equation*}
        \begin{split}
            \mathbb{E} \max_{x\in\mathcal{X}}\|\nabla \Phi_S(x)- \nabla \Phi(x)\|
            =
            \mathbb{E} \max_{x\in\mathcal{X}}\Big\|\frac{1}{n}\sum_{i=1}^n \nabla_x f(x,y_S^*(x);\xi_i)- \mathbb{E}\nabla_x f(x,y^*(x);\xi_i)\Big\|
        \end{split}
    \end{equation*}
    where 
    \begin{equation}
        y^*_S(x)\triangleq \argmax_{y\in\mathcal{Y}} F_S(x,y),
        \quad
        y^*(x)\triangleq \argmax_{y\in\mathcal{Y}} F(x,y)
    \end{equation}
    The primal function $\Phi_S$ is not in the form of averaging over $n$ samples, and thus existing analysis for the minimization problem is not directly applicable.
    In addition, as the optimal point $y^*_S(x)$ differs from $y^*(x)$, such difference brings in an additional error term. In the NC-SC case, the error is upper bounded by $\cO(n^{-1/2})$, which is the same scale as the error from establishing uniform convergence on $x$. Thus, the final uniform convergence bound established in Theorem \ref{thm:NC-SC Agnostic} is of the same order as that for minimization problem~\citep{mei2018landscape,davis2022graphical} except for an additional dependence on the condition number $\kappa$.
\item In the NC-C case, since there may exist multiple maximizers, we have 
\begin{equation}
    y^* \in \mathcal{Y}^*=\textrm{argmax}_{y\in\mathcal{Y}} \mathbb{E}f(x,y;\xi),
    \quad
    y_S^* \in \mathcal{Y}_S^*=\textrm{argmax}_{y\in\mathcal{Y}} \frac{1}{n}\sum_{i=1}^n f(x,y;\xi_i).
\end{equation}
Thus, the distance between $y^*$ and $y_S^*$ may not be well-defined.
Instead, we bound the distance between $\hat y_S^*(x)\triangleq \argmax_{y\in\mathcal{Y}} F_S(x,y)-\frac{\nu}{2}\|y\|^2$ and $\hat y^*(x)\triangleq \argmax_{y\in\mathcal{Y}} F(x,y)-\frac{\nu}{2}\|y\|^2$ for a small regularization parameter $\nu=\cO(n^{-1/2})$. The distance can be controlled by $\cO(n^{-1/4})$. Thus, the sample complexity for achieving $\eps$-uniform convergence for the NC-C case is large than that of the NC-SC case. We leave it for future investigation to see if one could achieve smaller sample complexity in the NC-C case via a better characterization of the extra error brought in by $y$ in the NC-C setting.
\end{enumerate}

\section{Algorithmic Stability and Generalization Bounds}
Notice that the uniform convergence in Theorems \ref{thm:NC-SC Agnostic} and \ref{thm:NC-C Agnostic} has a dependence on the dimension $d$. It becomes less meaningful for high-dimensional problems \citep{lei2022stability,feldman2019high}. It remains interesting to build dimension-independent generalization results utilizing the special structure of the algorithms. In this section, we investigate the generalization performance of specific algorithms for nonconvex stochastic minimax optimization problems utilizing stability arguments.

\subsection{Stability and Generalization}
Existing literature on stability arguments in minimax optimization often rely on stability notions based on function values \citep{farnia2021train,lei2021stability,zhang2021generalization}. In order to derive bounds on the generalization in terms of primal stationarity, we introduce the following novel notions of uniform stability on gradients of the primal function, called \textit{uniform primal stability}.
\begin{definition}[Uniform Primal Stability]
    \label{defn:uniform stability}    
    A randomized algorithm $\mathcal{A}$ is $\delta$-uniformly primal stable if for every two neighboring dataset $S, S'$ which differ in only one sample, for every $\xi\in\Xi$ we have
    \begin{equation}
        \begin{split}
            \sup_\xi\mathbb{E}_\mathcal{A}
            \autonorm{\nabla f\autopar{\mathcal{A}_x(S),y^*\autopar{\mathcal{A}_x(S)};\xi}
            -
            \nabla f\autopar{\mathcal{A}_x(S'), y^*\autopar{\mathcal{A}_x(S')};\xi}}^2\leq\delta^2,
        \end{split}
    \end{equation}
    where $\nabla f=(\nabla_x f, \nabla_y f)^\top$ denotes the full gradient.
\end{definition}

The following theorem connects stability and generalization in minimax optimization problems. We defer the proof to Appendix \ref{apdx:NCSC stability gen}.

\begin{theorem}[Stability and Generalization, NC-SC]
    \label{thm:stability and generalization}
    Let $\mathcal{A}$ be a $\delta$-uniformly primal stable algorithm. For any function $f$ satisfying Assumption \ref{as:ncsc-ncc} with $\mu>0$, we have
    \begin{equation}
        \begin{split}
            \mathbb{E}_{\mathcal{A},S}\autonorm{\nabla\Phi\autopar{\mathcal{A}_x(S)}-\nabla\Phi_S\autopar{\mathcal{A}_x(S)}}
            \leq
            (1+\kappa)\autopar{4\delta+\frac{ G}{\sqrt{n}}}.
        \end{split}
    \end{equation}
\end{theorem}
To the best of our knowledge, this is the first result that connects uniformly stable algorithms and generalization errors in minimax optimization measured by primal stationarity. As a comparison, in the minimization case, \citet[Theorem 2]{lei2022stability} proved that the gap between the empirical and population gradients is $\mathcal{O}\autopar{\delta+\frac{1}{\sqrt{n}}}$, while Theorem \ref{thm:stability and generalization} has an additional dependence on the condition number $\kappa$ that comes from the minimax structure.

In the NC-C case, the uniform primal stability is less meaningful as $y^*(\cdot)$ is not well-defined. Instead, we use the following notion of \textit{uniform primal argument stability}.

\begin{definition}[Uniform Primal Argument Stability]
    A randomized algorithm $\mathcal{A}$ is $\delta$-uniformly primal argument stable if for every two dataset $S, S'$ which differ in only one sample,
    \begin{equation*}
        \begin{split}
            \mathbb{E}_\mathcal{A}\autonorm{\mathcal{A}_x(S)-\mathcal{A}_x(S')}^2\leq\delta^2.
        \end{split}
    \end{equation*}
\end{definition}

The following theorem connects argument stability and generalization in NC-C case, measured by primal Moreau envelope stationarity.
\begin{theorem}[Stability and Generalization, NC-C]
    \label{thm:stability and generalization NCC}
    Let $\mathcal{A}$ be a $\delta$-uniformly primal argument stable algorithm. For any function $f$ satisfying Assumption \ref{as:ncsc-ncc} with $\mu=0$, we have
    \begin{equation}
        \begin{split}
            \mathbb{E}_{\mathcal{A},S}\autonorm{\nabla\Phi^{1/(2L)}\autopar{\mathcal{A}_x(S)}-\nabla\Phi_S^{1/(2L)}\autopar{\mathcal{A}_x(S)}}
            \leq 
            \mathcal{O}\autopar{\delta^{1/6}+n^{-1/8}}.
        \end{split}
    \end{equation}
\end{theorem}
We defer the proof to Appendix \ref{apdx:NCC stability gen}. Note that the analysis also leverages the idea of adding regularization to create a surrogate NC-SC problem, as we did in Section \ref{sec:uniform_NCC}. This result yields the relationship between stability and generalization in NC-C problems measured by primal stationarity. Different from the minimization case, the perturbation on the dataset incurs errors on both the function gradients and the dual maximizers, which requires more careful analysis to derive the final generalization bound. With Theorems \ref{thm:stability and generalization} and \ref{thm:stability and generalization NCC}, to obtain the generalization bounds of algorithms designed for NC-SC and NC-C minimax optimization problems, it suffices to derive the stability of specific algorithms. 

\subsection{Generalization of SGDA}
\label{sec:gen_SGDA}
In this subsection, we study the generalization bounds of the classical \textit{stochastic gradient descent ascent} (SGDA) for minimax optimization problems in both NC-SC and NC-C cases. Recall the procedures of SGDA: in each iteration $t$,
\begin{equation}
    \begin{cases}
        x_{t+1}=\proj_\mathcal{X}\autopar{x_t-\alpha_t^x\nabla_x f(x_t, y_t; \xi_t)},\\
        y_{t+1}=\proj_\mathcal{Y}\autopar{y_t+\alpha_t^y\nabla_y f(x_t, y_t; \xi_t)},
    \end{cases}
\end{equation}
where $\autopar{\alpha_t^x,\alpha_t^y}$ are the stepsizes. \citet{farnia2021train} investigated the $\delta$-stability of SGDA. Together with Theorems \ref{thm:stability and generalization} and \ref{thm:stability and generalization NCC}, we have the following generalization errors in NC-SC and NC-C cases, respectively.

\begin{corollary}[Generalization of SGDA, NC-SC]
    \label{thm:NC-SC SGDA Specific}
    Assume the function $f$ is NC-SC as defined in Assumption \ref{as:ncsc-ncc} with $\mu>0$, then if we run SGDA for $T$ iterations with stepsize $\autopar{\alpha_t^x,\alpha_t^y}=\autopar{\frac{c}{t},\frac{cr^2}{t}}$ for some constant $c>0$ and $1\leq r<\kappa$, let $\zeta_1=(cL(r+1)+1)^{-1}$, we have
    \begin{equation}
        \begin{split}
            \EE_{S,\mathcal{A}}\autonorm{\nabla\Phi\autopar{\mathcal{A}_x(S)}-\nabla\Phi_S\autopar{\mathcal{A}_x(S)}} 
            \leq
            \mathcal{O}\autopar{\kappa^{1+\zeta_1}
            \autopar{\frac{T^{1-\zeta_1}}{n}+\frac{1}{\sqrt{n}}}},
        \end{split}
    \end{equation}
    where $(\mathcal{A}_x(S), \mathcal{A}_y(S))=(x_T, y_T)$ is the output of SGDA.
\end{corollary}

\begin{corollary}[Generalization of SGDA, NC-C]
    \label{thm:NC-C SGDA Specific}
    Assume the function $f$ is NC-C as defined in Assumption \ref{as:ncsc-ncc} with $\mu=0$, then if we run SGDA for $T$ iterations with stepsize $\max\{\alpha_t^x,\alpha_t^y\}\leq\frac{c}{t}$ for some constant $c>0$, let $\zeta_2=(cL+1)^{-1}$ then we have
    \begin{equation}
        \begin{split}
            \EE_{S,\mathcal{A}}\autonorm{\nabla\Phi^{1/(2L)}\autopar{\mathcal{A}_x(S)}-\nabla\Phi_S^{1/(2L)}\autopar{\mathcal{A}_x(S)}}
            \leq
            \mathcal{O}\autopar{\autopar{\frac{T^{1-\zeta_2}}{n}}^{1/6} + \autopar{\frac{1}{n}}^{1/8}},
        \end{split}
    \end{equation}
    where $(\mathcal{A}_x(S), \mathcal{A}_y(S))=(x_T, y_T)$ is the output of SGDA.
\end{corollary}

The proof relies on the stability results in \citet{farnia2021train}, which we defer to Appendix \ref{apdx:NCSC SGDA}. Compared to the generalization bounds in Theorems \ref{thm:NC-SC Agnostic} and \ref{thm:NC-C Agnostic} that use uniform convergence, the generalization bounds of SGDA avoid the dependence on the dimension $d$. However, the dependence on $n$ of generalization bounds of SGDA becomes worse compared to uniform convergence in the NC-C setting. 

\subsection{Generalization of Sampling-determined Algorithms}
\label{sec:gen_SDA}
In this subsection, we consider an extension of  \emph{sampling-determined algorithm (SDA)} class proposed in \citet{lei2022stability}. For completeness, we present its definition below.
\begin{definition}[Sampling-determined Algorithm]
\label{defn:sampling-determined_alg}
    Let $\mathcal{A}$ be an algorithm that randomly chooses an index sequence $I(\mathcal{A})=\{i_t\}$ from the dataset to build stochastic gradients. 
    We say $\mathcal{A}$ is \emph{sampling-determined} if its output is independent of the sample $\xi_i$ for any $i\notin I(\mathcal{A})$.
\end{definition}

SDA covers a wide range of algorithms, including classical SGDA, stochastic extragradient, and some adaptive variants of SGDA~\citep{yang2022nest} in minimax optimization literature. \citet{lei2022stability} derives $\delta$-stability of SDA leveraging its sampling-determined property. Following their techniques and combining with Theorems \ref{thm:stability and generalization} and \ref{thm:stability and generalization NCC}, we obtain the following generalization bounds for SDA in both NC-SC and NC-C scenarios.

\begin{corollary}[Generalization of SDA, NC-SC]
    \label{thm:NC-SC SDA}
    Assume the function $f$ is NC-SC as defined in Assumption \ref{as:ncsc-ncc} with $\mu>0$. If we run a SDA algorithm $\mathcal{A}$ for $T$ iterations,  we have
    \begin{equation}
        \begin{split}
            \EE_{S,\mathcal{A}}\autonorm{\nabla\Phi\autopar{\mathcal{A}_x(S)}-\nabla\Phi_S\autopar{\mathcal{A}_x(S)}}
            \leq
            \mathcal{O}\autopar{\kappa \autopar{\sqrt{\frac{T}{n}}+\frac{1}{\sqrt{n}}}}
        \end{split}
    \end{equation}
\end{corollary}
Compared with Corollary \ref{thm:NC-SC SGDA Specific}, the generalization bound of SDA does not require specific stepsizes and applies to a wider class of algorithms. 
\begin{corollary}[Generalization of SDA, NC-C]
    \label{thm:NC-C SDA}
    Assume the function $f$ is NC-C as defined in Assumption \ref{as:ncsc-ncc} with $\mu=0$. If we run a SDA algorithm $\mathcal{A}$ for $T$ iterations, we have
    \begin{equation}
        \begin{split}
            \EE_{S,\mathcal{A}}\autonorm{\nabla\Phi^{1/(2L)}\autopar{\mathcal{A}_x(S)}-\nabla\Phi_S^{1/(2L)}\autopar{\mathcal{A}_x(S)}}
            \leq
            \mathcal{O}\autopar{\autopar{\frac{T}{n}}^{1/12} +
            \autopar{\frac{1}{n}}^{1/8}}.
        \end{split}
    \end{equation}
\end{corollary}
Compared with Corollary \ref{thm:NC-C SGDA Specific}, the generalization bound of SDA algorithm has a worse dependence on sample size $n$. Due to the $T/n$ term in Corollaries \ref{thm:NC-SC SDA} and \ref{thm:NC-C SDA}, to achieve good generalization, SDA should have less than one pass of the dataset. On the other hand, SGDA may use the stepsize to control $\zeta$ and can do multiple pass over the dataset. 

\section{Conclusion and Future Directions}
\label{sec:conclusion}
We take an initial step toward understanding the generalization performances of NC-SC and NC-C minimax problems measured by the first-order stationarity from both uniform convergence and stability argument perspectives. 
Several future directions are worth further investigation. It remains interesting to see whether we can improve the uniform convergence and stability results under the NC-C setting, particularly the dependence on sample size $n$. Another possible direction is to investigate the generalization performances for specific applications. Some studies in stochastic minimization show that specific machine learning models (e.g., generalized linear models) enjoy dimension-free uniform convergence bounds \citep{amir2022thinking,davis2022graphical}. It would be interesting to see if such dimension-free uniform convergence property also holds for some minimax applications.

\clearpage

\bibliographystyle{plainnat}
\bibliography{ref}

\clearpage
\appendix
\section{Additional Definitions and Tools}
\label{apdx_sec:notations}
For convenience, we summarize the notations commonly used throughout the paper.
\begin{itemize}[leftmargin = 2em]
    \item Population minimax problem and its primal function\footnote{Another commonly used convergence criterion in minimax optimization is the \textit{first-order stationarity of $F$}, i.e., $\autonorm{\nabla_x F}\leq\epsilon$ and $\autonorm{\nabla_y F}\leq\epsilon$ (or its corresponding gradient mapping) \citep{lin2020gradient,xu2020unified}. We refer readers to \cite{lin2020gradient,yang2022faster} for a thorough comparison of these two measurements. In this paper, we always stick to the convergence measured by the stationarity of the primal function. }
    \begin{equation*}
        \begin{split}
            & F(x,y)\triangleq \E_\xi f(x,y;\xi),\quad
            \Phi(x)\triangleq \max_{y\in\mathcal{Y}} F(x,y),\quad
            y^*(x)\triangleq \argmax_{y\in\mathcal{Y}} F(x,y).
        \end{split}
    \end{equation*}
    \item Empirical minimax problem and its primal function
    \begin{equation*}
        F_S(x,y)\triangleq \frac{1}{n}\sum_{i=1}^nf(x,y;\xi_i),\quad
        \Phi_S(x)\triangleq \max_{y\in\mathcal{Y}} F_S(x,y),\quad
        y_S^*(x)\triangleq \argmax_{y\in\mathcal{Y}} F_S(x,y).
    \end{equation*}
    \item Moreau envelope and corresponding proximal point:
    \begin{equation*}
        \begin{split}
            & \Phi^\lambda(x)\triangleq\min_{z\in\mathcal{X}}\autobigpar{\Phi(z)+\frac{1}{2\lambda}\autonorm{z-x}^2},\quad
            \prox_{\lambda \Phi}(x)\triangleq\argmin_{z\in\mathcal{X}}\autobigpar{\Phi(z)+\frac{1}{2\lambda}\autonorm{z-x}^2},\\
            & \Phi_S^\lambda(x)\triangleq\min_{z\in\mathcal{X}}\autobigpar{\Phi_S(z)+\frac{1}{2\lambda}\autonorm{z-x}^2},\quad
            \prox_{\lambda \Phi_S}(x)\triangleq\argmin_{z\in\mathcal{X}}\autobigpar{\Phi_S(z)+\frac{1}{2\lambda}\autonorm{z-x}^2}. 
        \end{split}
    \end{equation*}
    \item $\mathcal{G}_\Phi(x)$: gradient mapping (generalized gradient) of a function $\Phi$.
    \item $\autonorm{\cdot}$: $\ell_2$-norm.
    \item $\nabla f=\autopar{\nabla_x f, \nabla_y f}$: the gradient of a function $f$.
    \item $\proj_\mathcal{X}(x')$: the projection operator. 
    \item $\mathcal{A}(S)\triangleq\autopar{\mathcal{A}_x(S),\mathcal{A}_y(S)}$: the output of an algorithm $\mathcal{A}$ on the empirical minimax problem \eqref{eq:FS_problem} with dataset $S$.
    \item NC~/~WC: nonconvex, weakly convex.
    \item NC-SC~/~NC-C: nonconvex-(strongly)-concave.
    \item SOTA: state-of-the-art.
    \item $d$: dimension number of $\mathcal{X}$.
    \item $\kappa$: condition number $\frac{L}{\mu}$, $L$: Lipschitz smoothness parameter, $\mu$: strong concavity parameter.
    \item $\tilde{O}(\cdot)$ hides poly-logarithmic factors.
    \item $f=\Omega\autopar{g}$ if $f(x)\geq cg(x)$ for some $c>0$ and nonnegative functions $f$ and $g$.
    \item We say a function $g:\mathcal{X}\rightarrow\mathbb{R}$ is convex if $\forall\ x_1, x_2\in\mathcal{X}$ and $p\in[0,1]$, we have $g(px_1+(1-p)x_2)\geq pg(x_1)+(1-p)g(x_2)$.
    \item A function $h:\mathcal{X}\rightarrow\mathbb{R}$ is $L$-smooth\footnote{Here the smoothness definition for single-variable functions is subtly different from that of two-variable functions in Definition \ref{defn:L_smooth_minimax}, so we list it here for completeness.} if $h$ is continuously differentiable in $\mathcal{X}$ and there exists a constant $L>0$ such that $\autonorm{\nabla h(x_1)-\nabla h(x_2)}\leq L\|x_1-x_2\|$ holds  for any $x_1, x_2$.
\end{itemize}

For completeness, we introduce the definition of a sub-Gaussian random variable and related lemma, which are important tools in the analysis.

\begin{definition}[Sub-Gaussian Random Variable]
    A random variable $\eta$ is a zero-mean sub-Gaussian random variable with variance proxy $\sigma_\eta^2$ if $\EE \eta = 0$ and either of the following two conditions hold:
    \begin{align*}
    (a) ~\EE \automedpar{\exp(s \eta)}\leq \exp\autopar{\frac{\sigma_\eta^2 s^2}{2}} \text{ for any $s\in\RR$;}  \quad (b)~ \mathbb{P}(|\eta|\geq t)\leq 2\exp\autopar{-\frac{t^2}{2\sigma_\eta^2}} \text{ for any $t>0$.}
    \end{align*}
\end{definition}
We use the following McDiarmid's inequality to show that a random variable is sub-Gaussian.
\begin{lemma}[McDiarmid’s inequality]
    \label{lm:McDiarmid_inequality}
Let $\eta_1,\ldots,\eta_n\in\RR$ be independent random variables. Let $h:\RR^n\rightarrow\RR$ be any function     with the $(c_1,\ldots,c_n)$-bounded differences property: for every $i=1,\ldots,n$ and every $(\eta_1,\ldots,\eta_n)$, and $(\eta_1^\prime,\dots,\eta_n^\prime)$ that differ only in the $i$-th coordinate ($\eta_j= \eta_j^\prime$ for all $j\not=i$), we have 
\begin{equation*}
\autoabs{h(\eta_1,\ldots,\eta_n) - h(\eta_1^\prime,\dots,\eta_n^\prime)}\leq c_i.
\end{equation*}
For any $t>0$, it holds that 
\begin{equation*}
\mathbb{P}\autopar{\autoabs{h(\eta_1,\ldots,\eta_n)-\EE h(\eta_1,\ldots,\eta_n)} \geq t}\leq 2\exp\autopar{-\frac{2t^2}{\sum_{i=1}^n c_i^2}}.
\end{equation*}
\end{lemma}

\begin{lemma}[Properties of $\Phi$ and $\Phi^\lambda$, Full Version]
    \label{lm:properties_Phi_envelope_full}
    In the NC-SC setting ($\mu>0$), both $\Phi(x)$ and $\Phi_S(x)$ are $\tilde{L}\triangleq L(1+\kappa)$-smooth with the condition number $\kappa\triangleq L/\mu$, both $y^*(x)$ and $y_S^*(x)$ are $\kappa$-Lipschitz continuous and $\nabla\Phi(x)   = \nabla_x F(x,y^*(x)), \nabla\Phi_S(x) = \nabla_xF_S\autopar{x,y_S^*(x)}$. In the NC-C setting ($\mu=0$), the primal function $\Phi$ is $L$-weakly convex, and its its Moreau envelope $\Phi^{\lambda}(x)$ is differentiable, Lipschitz smooth, also 
    \begin{equation}
        \nabla\Phi^{\lambda}(x)=\lambda^{-1}(x-\hat{x}),
        \quad
        \autonorm{\nabla\Phi^{\lambda}(x)}\geq\dist\autopar{0,\partial\Phi(\hat{x})},
    \end{equation}
    where $\hat{x}=\prox_{\lambda\Phi}(x)$ and $0<\lambda<1/L$.
\end{lemma}

For completeness, we formally define the stationary point here. Note that the generalized gradient is defined on $\mathcal{X}$ while the Moreau envelope is defined on the whole domain $\mathbb{R}^d$.

\begin{definition}[Stationary Point]
\label{defn:stationary_point}
    Let $\epsilon>0$, for an $\tilde{L}$-smooth function $\Phi:\mathcal{X}\rightarrow\mathbb{R}$, we call a point $x$ an $\epsilon$-stationary point of $\Phi$ if $\autonorm{\mathcal{G}_\Phi(x)}\leq\epsilon$, where $\mathcal{G}_\Phi$ is the gradient mapping (or generalized gradient) defined as $\mathcal{G}_\Phi(x)\triangleq \tilde{L}\autopar{x-\proj_\mathcal{X}\autopar{x-(1/\tilde{L})\nabla\Phi(x)}}$; for an $L$-weakly convex function $\Phi$, we say a point $x$ an $\epsilon$-(nearly)-stationary point of $\Phi$ if $\autonorm{\nabla\Phi^{1/(2L)}(x)}\leq\epsilon$.
\end{definition}

\section{Proof of Theorem \ref{thm:NC-SC Agnostic}}
\label{apdx_sec:proof_NCSC_unif_convergence}
\begin{proof}
To derive the desired generalization bounds, we take an $\upsilon$-net $\{x_k\}_{k=1}^Q$ on $\mathcal{X}$ so that there exists a $k\in\autobigpar{1,\cdots,Q}$ for any $x\in\mathcal{X}$ such that $\|x-x_k\|\leq \upsilon$. Note that such $\upsilon$-net exists with  $Q=\cO(\upsilon^{-d})$ for compact $\mathcal{X}$~\citep{kleywegt2002sample}. Utilizing the definition of the $\upsilon$-net, we have
\begin{equation}
    \begin{split}
        &
    \EE \max_{x\in\mathcal{X}}\autonorm{\nabla \Phi_S(x)- \nabla\Phi(x)} \\
\leq\ &
    \EE \max_{x\in\mathcal{X}} \automedpar{\autonorm{\nabla \Phi_S(x)- \nabla\Phi_S(x_k)} + \autonorm{\nabla \Phi_S(x_k)- \nabla\Phi(x_k)} + \autonorm{\nabla \Phi(x_k)- \nabla\Phi(x)}}\\
\leq\ & 
    \EE \max_{k\in[Q]} \|\nabla \Phi_S(x_k)- \nabla\Phi(x_k)\| + 2L(1+\kappa) \upsilon,
    \end{split}
\end{equation}
where the last inequality holds as $\Phi$ and $\Phi_S$ are $L(1+\kappa)$-smooth following Lemma \ref{lm:properties_Phi_envelope}. For any $s>0$, we have 
\begin{equation}
\label{eq:u-net-transformation}
\begin{aligned}
&
    \exp\Big( s\EE \max_{x\in\mathcal{X}}\|\nabla \Phi_S(x)- \nabla\Phi(x)\|\Big)\\
\leq\ &
    \exp\Big( s\Big[\EE \max_{k\in[Q]} \|\nabla \Phi_S(x_k)- \nabla\Phi(x_k)\| + 2L(1+\kappa) \upsilon \Big]\Big) \\
\leq\ &
    \EE \max_{k\in[Q]} \exp\Big( s\Big[ \|\nabla \Phi_S(x_k)- \nabla\Phi(x_k)\| + 2L(1+\kappa) \upsilon \Big]\Big) \\
\leq\ &
    \EE \sum_{k\in[Q]} \exp\Big( s\Big[ \|\nabla \Phi_S(x_k)- \nabla\Phi(x_k)\| + 2L(1+\kappa) \upsilon \Big]\Big) \\
=\ &
    \sum_{k\in[Q]} \EE  \exp\Big( s\Big[ \|\nabla \Phi_S(x_k)- \nabla\Phi(x_k)\| + 2L(1+\kappa) \upsilon \Big]\Big),
\end{aligned}
\end{equation}
where the second inequality uses Jensen's inequality and monotonicity of exponential function, and the third inequality uses summation over $k\in[Q]$ to handle the dependence issue, i.e., the $x_k$ in the last line is independent of $S$. We use the exponential function as an intermediate step so that the final sample complexity depends on $\log(Q)$ rather than $Q$, which is of order $\cO(\upsilon^{-d})$.  Without loss of generality, selecting $\upsilon$ such that $2L(1+\kappa) \upsilon =\frac{\eps}{2}$, we have
\begin{equation}
\label{eq:generalization_agnostic}
\begin{aligned}
&
    \EE \max_{x\in\mathcal{X}}\|\nabla \Phi_S(x)- \nabla\Phi(x)\| \\
\leq\ &
    \frac{1}{s} \log\Bigg(\sum_{k\in [Q]}\EE\exp\autopar{s\automedpar{\autonorm{\nabla{\Phi}(x_k)-\nabla{\Phi_S}(x_k)}-\EE\autonorm{\nabla{\Phi}(x_k)-\nabla{\Phi_S}(x_k)}}}\\
    &\qquad\qquad\qquad\qquad\qquad\qquad
    \cdot\exp\autopar{s\EE\autonorm{\nabla{\Phi}(x_k)-\nabla{\Phi_S}(x_k)}}\exp\autopar{\frac{s\eps}{2}}\Bigg).  
\end{aligned}
\end{equation}
To upper bound $\EE\|\nabla{\Phi}(x_k)-\nabla{\Phi_S}(x_k)\|$, we use the following observation. Define $y_{S^{(i)}}^*(x) \triangleq \argmax_{y\in\mathcal{Y}} F_{S^{(i)}}(x,y)$ where $S=\{\xi_i\}_{i=1}^n$, $S^{(i)} = \{\xi_1,\ldots, \xi_{i-1}, \xi_i^\prime, \xi_{i+1},\ldots, \xi_n\}$ and $\xi_i^\prime$ is i.i.d. from $\xi_i$. Since $x$ is independent of $S$ or $S^{(i)}$ for any $i$, by Danskin's theorem, we have
\begin{equation}
    \begin{split}
        &\ \mathbb{E}\|\nabla{\Phi}(x)-\nabla{\Phi_S}(x)\| 
        =
        \EE\autonorm{\EE_\xi \nabla_x f(x,y^*(x);\xi) -\frac{1}{n}\sum_{i=1}^n \nabla_x f(x,y_S^*(x);\xi_i)}\\
        =\ &
        \EE\Bigg\|\EE_\xi \nabla_x f(x,y^*(x);\xi)-\frac{1}{n}\sum_{i=1}^n \nabla_x f(x,y^*(x);\xi_i)\\
        &\qquad\qquad\qquad
        +\frac{1}{n}\sum_{i=1}^n \nabla_x f(x,y^*(x);\xi_i) -\frac{1}{n}\sum_{i=1}^n \nabla_x f(x,y_S^*(x);\xi_i)\Bigg\|\\
        \leq\ &
        \EE\autonorm{\EE_\xi \nabla_x f(x,y^*(x);\xi)-\frac{1}{n}\sum_{i=1}^n \nabla_x f(x,y^*(x);\xi_i)}\\
        &\qquad\qquad\qquad
        +\EE\autonorm{\frac{1}{n}\sum_{i=1}^n \nabla_x f(x,y^*(x);\xi_i) -\frac{1}{n}\sum_{i=1}^n \nabla_x f(x,y_S^*(x);\xi_i)}\\
        \leq\ &
        \EE\autonorm{\EE_\xi \nabla_x f(x,y^*(x);\xi)-\frac{1}{n}\sum_{i=1}^n \nabla_x f(x,y^*(x);\xi_i)}
        +
        L\autonorm{y^*(x)-y_S^*(x)}\\
        \leq\ &
        \sqrt{\frac{\text{Var}(\nabla_x f)}{n}}+L\autonorm{y^*(x)-y_S^*(x)},
    \end{split}
\end{equation}
where $\text{Var}(\nabla_x f)$ is the variance of $\nabla_x f(\cdot,\cdot;\xi)$ and the second inequality holds by smoothness of $f$. Since the variance is upper bounded by the second moment:
\begin{equation}
    \text{Var}(\nabla_x f)\leq\mathbb{E}\autonorm{\nabla_x f(x,y^*(x);\xi)}^2\leq G^2,
\end{equation}
it further holds that
\begin{equation}
    \label{eq:NCSC_expectation_gradient_difference}
    \mathbb{E}\|\nabla{\Phi}(x)-\nabla{\Phi_S}(x)\| 
    \leq
    \frac{G}{\sqrt{n}}+L\autonorm{y^*(x)-y_S^*(x)}.
\end{equation}
To derive an upper bound on $\autonorm{y^*(x)-y_S^*(x)}$, we first bound $\autonorm{y_{S^{(i)}}^*(x)-y_S^*(x)}$ and utilize the stability argument. Since $f(x,y;\xi)$ is $\mu$-strongly concave in $y$ for any $x$ and $\xi$ and $y_S^*(x)$ is the maximizer of $F_S(x,\cdot)$, we have \begin{equation}
    \label{eq:F_S_strong_concavity}
    \autopar{-F_S\autopar{x,y_{S^{(i)}}^*(x)}} - \autopar{-F_S\autopar{x,y_S^*(x)}} 
    \geq
    \frac{\mu}{2}\Big\|y_{S^{(i)}}^*(x)-y_S^*(x)\Big\|^2,
\end{equation}
On the other hand, we have 
\begin{equation*}
\begin{aligned}
& 
    F_S(x,y_S^*(x)) -F_S(x,y_{S^{(i)}}^*(x))   \\ 
= &
    F_{S^{(i)}}(x,y_S^*(x)) - F_{S^{(i)}}(x,y_{S^{(i)}}^*(x))\\
    &\qquad + \frac{1}{n}\Big[ f(x, y_S^*(x);\xi_i) -f(x, y_{S^{(i)}}^*(x);\xi_i) + f(x, y_{S^{(i)}}^*(x);\xi_i^\prime) -f(x, y_S^*(x);\xi_i^\prime)\Big]\\
\leq  & 
    F_{S^{(i)}}(x,y_S^*(x)) - F_{S^{(i)}}(x,y_{S^{(i)}}^*(x))\\
    &\qquad + \frac{1}{n} \Big|f(x, y_{S^{(i)}}^*(x);\xi_i) -f(x, y_S^*(x);\xi_i)| + \frac{1}{n} |f(x, y_{S^{(i)}}^*(x);\xi_i^\prime) -f(x, y_S^*(x);\xi_i^\prime)\Big|\\
\leq &
    \frac{2G}{n}\big\|y_{S^{(i)}}^*(x)-y_S^*(x)\big\|,
\end{aligned}
\end{equation*}
where the last inequality holds by Lipschitz continuity and the optimality of $y_{S^{(i)}}^*(x)$.
Combined with \eqref{eq:F_S_strong_concavity}, it holds that 
\begin{equation*}
\Big\|y_{S^{(i)}}^*(x)-y_S^*(x)\Big\|\leq \frac{4G}{\mu n}.
\end{equation*}
In addition, we have
\begin{equation}
\label{eq:ERM_generalization}
\begin{aligned}
& 
    \EE [F(x,y^*(x))-F(x,y_S^*(x))] \\ 
=\ &
    \EE \automedpar{F(x,y^*(x))- F_S(x,y^*(x))} +\EE\automedpar{F_S(x, y^*(x)) -  F_S(x,y_S^*(x))}\\
    &\qquad\qquad\qquad + \EE \automedpar{F_S(x,y_S^*(x)) - F(x,y_S^*(x))} \\
\leq & 
    \EE \automedpar{F_S(x,y_S^*(x)) - F(x,y_S^*(x))} \\
=\ & 
    \EE \automedpar{\frac{1}{n}\sum_{i=1}^n f(x,y_S^*(x);\xi_i) - \frac{1}{n}\sum_{i=1}^n \EE_\xi f(x,y_S^*(x);\xi)} \\
=\ &
    \EE \automedpar{\frac{1}{n}\sum_{i=1}^n f(x,y_S^*(x);\xi_i) - \frac{1}{n}\sum_{i=1}^n \EE_{\xi_i} f(x,y_{S^{(i)}}^*(x);\xi_i)} \\
=\ &
    \EE \automedpar{\frac{1}{n}\sum_{i=1}^n f(x,y_S^*(x);\xi_i) - \frac{1}{n}\sum_{i=1}^n f(x,y_{S^{(i)}}^*(x);\xi_i)} \\
\leq\ &
    G \EE\autonorm{y_S^*(x) - y_{S^{(i)}}^*(x)}\\
\leq\ &
    \frac{4 G^2}{\mu n}
\end{aligned}
\end{equation}
where the first inequality holds as $y_S^*(x) = \argmax_{y\in\mathcal{Y}} F_S(x,y_S^*(x))$ and $\EE [F(x,y^*(x))- F_S(x,y^*(x))]=0$, the third equality holds as $y_S^*(x)$ and $y_{S^{(i)}}^*(x)$ are identical distributed and $y_{S^{(i)}}^*(x)$ is independent of $\xi$ by definition, the second inequality holds by Lipschitz continuity of $f$ on $y$, and the last inequality holds by plugging the upper bound on $\|y_S^*(x) - y_{S^{(i)}}^*(x)\|$. On the other hand, since $F(x,y)$ is strongly concave in $y$ and $y^*(x)=\argmax_{y\in\mathcal{Y}} F(x,y)$, it holds that
\begin{equation*}
F(x,y^*(x))-F(x,y_S^*(x)) \geq \frac{\mu}{2}\|y^*(x)-y_S^*(x)\|^2.
\end{equation*}
Therefore, we have
\begin{equation*}
\EE\|y^*(x)-y_S^*(x)\|\leq \sqrt{\frac{8G^2}{\mu^2 n}}.
\end{equation*}
Plugging into \eqref{eq:NCSC_expectation_gradient_difference}, it holds that
\begin{equation}
\label{eq:NCSC-expectation}
\EE\|\nabla{\Phi}(x_k)-\nabla{\Phi_S}(x_k)\| \leq  L\sqrt{\frac{8G^2}{\mu^2 n}}+\frac{G}{\sqrt{n}}.
\end{equation}
Next we show that $\|\nabla{\Phi}(x)-\nabla{\Phi_S}(x)\|-\EE \|\nabla{\Phi}(x)-\nabla{\Phi_S}(x)\|$ is zero-mean sub-Gaussian. Notice that for any $\xi_i^\prime$, we have
\begin{equation}
    \begin{split}
        &   
            \|\nabla{\Phi}(x)-\nabla{\Phi_S}(x)\| - \|\nabla{\Phi}(x)-\nabla{\Phi_{S^{(i)}}}(x)\| \\
        \leq\ &
            \|\nabla{\Phi_S}(x) -\nabla{\Phi_{S^{(i)}}}(x) \| \\
        =\ &
            \autonorm{\frac{1}{n}\sum_{j=1}^n \nabla_x f\autopar{x,y_S^*(x),\xi_j} 
            - 
            \frac{1}{n}\sum_{j\not=i}^n\nabla_xf\autopar{x,y_{S^{(i)}}^*(x),\xi_j} 
            - 
            \frac{1}{n}\nabla_xf\autopar{x,y_{S^{(i)}}^*(x),\xi_i^\prime}}\\
        \leq\ &
            L \autonorm{y_{S^{(i)}}^*(x)-y_S^*(x)} + \frac{1}{n} \autonorm{\nabla_xf(x,y_{S^{(i)}}^*(x);\xi_i^\prime) - \nabla_xf(x,y_{S^{(i)}}^*(x);\xi_i)}\\
        \leq\ &
            \frac{4L G/\mu+ 2 G}{n},
    \end{split}
\end{equation}
where the first inequality uses triangle inequality, the first equality uses the definition of $\Phi_S$ and $\Phi_{S^{(i)}}$, the third inequality uses the assumption that $G$ is the uniform upper bound of $\nabla f(x,y;\xi)$ on $\mathcal{X}\times\mathcal{Y}$ for any $\xi$. By McDiarmid's inequality (Lemma \ref{lm:McDiarmid_inequality}) and the definition of sub-Gaussian random variables, it holds that $\|\nabla{\Phi}(x_k)-\nabla{\Phi_S}(x_k)\|-\EE\|\nabla{\Phi}(x_k)-\nabla{\Phi_S}(x_k)\|$ is a zero-mean sub-Gaussian random variable with variance proxy $\sigma^2 \triangleq \autopar{2L G/\mu+ G}^2/n$. By the definition of zero-mean sub-Gaussian random variables, it holds that
\begin{equation}
\label{eq:NCSC_expecation_subgaussian}
\EE\exp(s[\|\nabla{\Phi}(x_k)-\nabla{\Phi_S}(x_k)\|-\EE\|\nabla{\Phi}(x_k)-\nabla{\Phi_S}(x_k)\|]) \leq \exp\autopar{\frac{s^2\sigma^2}{2}}.
\end{equation}
Plugging \eqref{eq:NCSC-expectation} and \eqref{eq:NCSC_expecation_subgaussian} into \eqref{eq:generalization_agnostic}, we have
\begin{equation}
    \EE \|\nabla \Phi_S(x)- \nabla\Phi(x)\|
\leq 
    \frac{log(Q)}{s} + \frac{s\sigma^2}{2} + L\sqrt{\frac{8G^2}{\mu^2 n}} + \frac{G}{\sqrt{n}} + \frac{\eps}{2}
\end{equation}
Minimizing the right-hand side over $s$, we have
\begin{equation}
    \begin{split}
        \EE \|\nabla \Phi_S(x)- \nabla\Phi(x)\|
        \leq\ &
        2\sqrt{\frac{log(Q) \sigma^2}{2}} + L\sqrt{\frac{8G^2}{\mu^2 n}} +\frac{G}{\sqrt{n}} + \frac{\eps}{2}\\
        =\ &
        \sqrt{\frac{2log(Q) \autopar{2L G/\mu+ G}^2}{n}} + L\sqrt{\frac{8G^2}{\mu^2 n}} +\frac{G}{\sqrt{n}} + \frac{\eps}{2}.
    \end{split}
\end{equation}
Recall that $Q = \cO(\upsilon^{-d})$ with $\upsilon = \epsilon/(4L(1+\kappa))$, thus $\log(Q) =\cO(d\log(4L(1+\kappa)\epsilon^{-1}))$, which verifies the first statement in the theorem. For the sample complexity, following the discussion on the performance measurement in Section \ref{sec:prelim}, it is easy to derive that it requires
\begin{equation}
    n = \cO\Big(2d\eps^{-2}(2L G/\mu+ G)^2\log(4L(1+\kappa)\epsilon^{-1})\Big)=\tilde \cO(d\kappa^2\eps^{-2})
\end{equation}
to guarantee that $\EE \|\nabla \Phi_S(x)- \nabla\Phi(x)\|\leq\eps$ for any $x\in\mathcal{X}$, which concludes the proof.
\end{proof}

\section{Proof of Theorem \ref{thm:NC-C Agnostic}}
\label{apdx_sec:proof_NCC_unif_convergence}

We first provide the proof of Lemma \ref{lemma:nc_sc_to_ncc}.

\begin{proof}
Since $F(x,y)$ is $L$-smooth, it is obvious that $F(x,y)-\frac{\nu}{2}\|y\|^2$ is $(L+\nu)$-smooth.
By \citet[Lemma 3]{thekumparampil2019efficient}, $\hat \Phi(x)$ is $(L+\nu)$-weakly convex in $x$. Therefore, $\hat \Phi(x)+\frac{1}{2\lambda}\|x-x'\|^2$ is $\autopar{\frac{1}{\lambda}- (L+\nu)}$-strongly convex in $x$ for any fixed $x'$. Denote $\hat y(x)\triangleq\argmax_{y\in\mathcal{Y}} F(x,y)-\frac{\nu}{2}\|y\|^2$, $y^*(x)\triangleq \argmax_{y\in\mathcal{Y}} F(x,y)$. It holds that
\begin{equation}
\label{eq:link}
\begin{aligned}
    &
    \frac{1}{2}(1/\lambda- (L+\nu))\|\prox_{\lambda \Phi}(x) - \prox_{\lambda \hat \Phi}(x)\|^2 \\
\leq\ &
    \hat \Phi(\prox_{\lambda \Phi}(x))+\frac{1}{2\lambda}\|\prox_{\lambda \Phi}(x)-x\|^2 - \hat \Phi(\prox_{\lambda \hat \Phi}(x))- \frac{1}{2\lambda}\|\prox_{\lambda\hat  \Phi}(x)-x\|^2\\
=\ & 
    F(\prox_{\lambda \Phi}(x),\hat y(\prox_{\lambda \Phi}(x))) - \frac{\nu}{2}\|\hat y(\prox_{\lambda \Phi}(x))\|^2 + \frac{1}{2\lambda}\|\prox_{\lambda \Phi}(x)-x\|^2\\
    & \qquad\qquad- 
    F(\prox_{\lambda \hat \Phi}(x),\hat y(\prox_{\lambda \hat \Phi}(x))) + \frac{\nu}{2}\|\hat y(\prox_{\lambda \hat \Phi}(x))\|^2 - \frac{1}{2\lambda}\|\prox_{\lambda\hat  \Phi}(x)-x\|^2\\
\leq\ &
    F(\prox_{\lambda \Phi}(x), y^*(\prox_{\lambda \Phi}(x))) + \frac{1}{2\lambda}\|\prox_{\lambda \Phi}(x)-x\|^2 - \frac{\nu}{2}\|\hat y(\prox_{\lambda \Phi}(x))\|^2\\
    & \qquad\qquad- 
    F(\prox_{\lambda \hat \Phi}(x),\hat y(\prox_{\lambda \hat \Phi}(x))) - \frac{1}{2\lambda}\|\prox_{\lambda\hat  \Phi}(x)-x\|^2  + \frac{\nu}{2}\|\hat y(\prox_{\lambda \hat \Phi}(x))\|^2\\
\leq\ &
    F(\prox_{\lambda \Phi}(x), y^*(\prox_{\lambda \Phi}(x))) + \frac{1}{2\lambda}\|\prox_{\lambda \Phi}(x)-x\|^2 - \frac{\nu}{2}\|\hat y(\prox_{\lambda \Phi}(x))\|^2\\
    & \qquad\qquad- 
    F(\prox_{\lambda \hat \Phi}(x), y^*(\prox_{\lambda \hat \Phi}(x))) - \frac{1}{2\lambda}\|\prox_{\lambda\hat  \Phi}(x)-x\|^2  + \frac{\nu}{2}\| y^*(\prox_{\lambda \hat \Phi}(x))\|^2\\
=\ &
    \Phi(\prox_{\lambda  \Phi}(x)) + \frac{1}{2\lambda}\|\prox_{\lambda \Phi}(x)-x\|^2 - \Phi(\prox_{\lambda \hat \Phi}(x))-\frac{1}{2\lambda}\|\prox_{\lambda \hat\Phi}(x)-x\|^2\\
    & \qquad\qquad+ 
    \frac{\nu}{2}\|y^*(\prox_{\lambda \hat \Phi}(x))\|^2-\frac{\nu}{2}\|\hat y(\prox_{\lambda \Phi}(x))\|^2\\
\leq\ & 
    \frac{\nu}{2}\|y^*(\prox_{\lambda \hat \Phi}(x))\|^2-\frac{\nu}{2}\|\hat y(\prox_{\lambda \Phi}(x))\|^2\\
\leq\ &
    \frac{\nu D_\mathcal{Y}}{2},
\end{aligned}
\end{equation}
where the first inequality holds by strong convexity of $\hat \Phi(z)+\frac{1}{2\lambda}\|z-x\|^2$ and optimality of $\prox_{\lambda \hat \Phi}(x)$ for $\min_{z\in\mathcal{X}}\hat \Phi(z)+\frac{1}{2\lambda}\|z-x\|^2$, the first equality holds by definition of $\hat \Phi$,  the second inequality holds by optimality of $y^*(\prox_{\lambda \Phi}(x)))=\argmax_{y\in\mathcal{Y}} F(\prox_{\lambda \Phi}(x), y)$, the third inequality holds by optimality of $\hat y(\prox_{\lambda \Phi}(x)))=\argmax_{y\in\mathcal{Y}} F(\prox_{\lambda \Phi}(x), y)-\frac{\nu}{2}\|y\|^2$, the second equality holds by definition of $\Phi$, the fourth inequality holds by optimality of $\prox_{\lambda \Phi}(x) = \argmin_{x\in\mathcal{X}} \{ \Phi(z)+\frac{1}{2\lambda}\|z-x\|^2\}$, the last inequality holds by the compactness of domain $\mathcal{Y}$.    
\end{proof}

Next, we demonstrate the proof of Theorem \ref{thm:NC-C Agnostic}.

\begin{proof}
By Lemma \ref{lemma:nc_sc_to_ncc}, we have
\begin{equation*}
\begin{aligned}
& 
    \|\prox_{\lambda \Phi}(x) - \prox_{\lambda \hat \Phi}(x)\|\leq \sqrt{\frac{\lambda\nu D_\mathcal{Y}}{1-\lambda(L+\nu)}};\\
& 
    \|\prox_{\lambda \Phi_S}(x) - \prox_{\lambda \hat \Phi_S}(x)\|\leq \sqrt{\frac{\lambda\nu D_\mathcal{Y}}{1-\lambda(L+\nu)}}.
\end{aligned}
\end{equation*}
To derive the desired uniform convergence, similar to the proof of Theorem \ref{thm:NC-SC Agnostic}, we take an $\upsilon$-net $\{x_k\}_{k=1}^Q$ on $\mathcal{X}$ so that there exists a $k\in\autobigpar{1,\cdots, Q}$ for any $x\in\mathcal{X}$ such that $\|x-x_k\|\leq \upsilon$. Note that such $\upsilon$-net exists with  $Q=\cO(\upsilon^{-d})$ for compact $\mathcal{X}$. We first decompose the error as the approximation error from NC-SC minimax problems to NC-C minimax problems. Then we utilize the $\upsilon$-net to address the dependence between $S$ and $\argmax_{x\in\mathcal{X}} \|\nabla  \Phi_S^\lambda(x)-\nabla \Phi^\lambda(x)\|$. First, note that

\begin{equation}
\label{eq:NCC-agnostic-key}
\begin{aligned}
&
    \EE\max_{x\in\mathcal{X}}\|\nabla  \Phi_S^\lambda(x)-\nabla \Phi^\lambda(x)\| \\
= & 
    \frac{1}{\lambda}\EE\max_{x\in\mathcal{X}}\|\prox_{\lambda \Phi_S}(x)-\prox_{\lambda \Phi}(x)\|\\
\leq &
    \frac{1}{\lambda}\EE\max_{x\in\mathcal{X}}\|\prox_{\lambda \Phi_S}(x)-\prox_{\lambda \hat \Phi_S}(x)\|+\|\prox_{\lambda \hat \Phi_S}(x)-\prox_{\lambda \hat \Phi}(x)\|\\
    &\qquad\qquad\qquad +\|\prox_{\lambda \hat \Phi}(x)- \prox_{\lambda \Phi}(x)\|\\
\leq &
    \frac{2}{\lambda}\sqrt{\frac{\lambda\nu D_\mathcal{Y}}{1-\lambda(L+\nu)}} + \frac{1}{\lambda}\EE\max_{x\in\mathcal{X}}\|\prox_{\lambda \hat \Phi_S}(x)-\prox_{\lambda \hat \Phi}(x)\|\\
\leq &
    \frac{2}{\lambda}\sqrt{\frac{\lambda\nu D_\mathcal{Y}}{1-\lambda(L+\nu)}} + \frac{1}{\lambda}\EE\max_{x\in\mathcal{X}}\big[\|\prox_{\lambda \hat \Phi_S}(x)-\prox_{\lambda \hat \Phi_S}(x_k)\| \\
    &\qquad\qquad\qquad\qquad\  
    + \|\prox_{\lambda \hat \Phi_S}(x_k) - \prox_{\lambda \hat \Phi}(x_k) \| + \|\prox_{\lambda \hat \Phi}(x_k)-\prox_{\lambda \hat \Phi}(x)\|\big]\\
\leq &
    2\sqrt{\frac{\nu D_\mathcal{Y}}{\lambda(1-\lambda(L+\nu))}} + \frac{1}{\lambda}\EE\max_{k\in[Q]}\|\prox_{\lambda \hat \Phi_S}(x_k) - \prox_{\lambda \hat \Phi}(x_k) \| + \frac{2\upsilon}{\lambda (1-\lambda(L+\nu))}\\
\leq &
    2\sqrt{\frac{\nu D_\mathcal{Y}}{\lambda(1-\lambda(L+\nu))}} + \frac{1}{\lambda s}\log\autopar{\sum_{k\in[Q]}\EE \exp\autopar{s\autonorm{\prox_{\lambda \hat \Phi_S}(x_k) - \prox_{\lambda \hat \Phi}(x_k)}}}\\
    &\qquad\qquad\qquad\qquad\qquad + \frac{2\upsilon}{\lambda (1-\lambda(L+\nu))},
\end{aligned}
\end{equation}
where the first and the third inequality use the triangle inequality, the second  inequality uses Lemma \ref{lemma:nc_sc_to_ncc} for $\Phi$ and $\Phi_S$,  $x_k$ is the closest point to $x$ in the $\upsilon$-net, the fourth inequality holds by $(1-\lambda(L+\nu))^{-1}$-Lipschitz continuity of proximal operator~\citep[Lemma 4.3]{davis2022graphical}  since $F(x,y)-\frac{\nu}{2}\|y\|^2$ is a $(L+\nu)$-smooth function, and the last inequality follows a similar argument in \eqref{eq:u-net-transformation}. All that remains is to bounding $\EE \exp\autopar{s\autonorm{\prox_{\lambda \hat \Phi_S}(x) - \prox_{\lambda \hat \Phi}(x)}}$ for $x\in\mathcal{X}$  that is independent of $S$. Notice that 
\begin{equation*}
\begin{aligned}
&
     \EE \exp\autopar{s\autonorm{\prox_{\lambda \hat \Phi_S}(x_k) - \prox_{\lambda \hat \Phi}(x_k)}} \\
=\ &
    \EE \exp\autopar{s\automedpar{\autonorm{\prox_{\lambda \hat \Phi_S}(x_k) - \prox_{\lambda \hat \Phi}(x_k) } -\EE \autonorm{\prox_{\lambda \hat \Phi_S}(x_k) - \prox_{\lambda \hat \Phi}(x_k)}}} \\ 
    &\qquad\qquad\qquad
    \cdot\exp\autopar{s\EE \autonorm{\prox_{\lambda \hat \Phi_S}(x_k) - \prox_{\lambda \hat \Phi}(x_k)}}\\
\end{aligned}
\end{equation*}
Next, we show that $\autonorm{\prox_{\lambda \hat \Phi_S}(x_k) - \prox_{\lambda \hat \Phi}(x_k) } -\EE \autonorm{\prox_{\lambda \hat \Phi_S}(x_k) - \prox_{\lambda \hat \Phi}(x_k)}$ is a zero-mean sub-Guassian random variable and $\EE \autonorm{\prox_{\lambda \hat \Phi_S}(x_k) - \prox_{\lambda \hat \Phi}(x_k)}$ is bounded.
Since $x_k$ is independent of $S$, it is sufficient to show an upper bound of the following term where $x\in\mathcal{X}$ is independent of $S$.
\begin{equation*}
\EE \|\prox_{\lambda \hat \Phi_S}(x) - \prox_{\lambda \hat \Phi}(x) \|.
\end{equation*}
Recall the definition that 
\begin{align}
& 
    \prox_{\lambda \hat \Phi}(x) = \argmin_{z\in\mathcal{X}}\Big\{ \max_{y\in\mathcal{Y}} \EE_\xi f(z,y;\xi) -\frac{\nu}{2}\|y\|^2+\frac{1}{2\lambda} \|z-x\|^2\Big\}, \label{eq:zyx}\\
&
    \prox_{\lambda \hat \Phi_S}(x) = \argmin_{z\in\mathcal{X}}\autobigpar{\max_{y\in\mathcal{Y}} \frac{1}{n}\sum_{i=1}^n \Big[f(z,y;\xi_i) -\frac{\nu}{2}\|y\|^2+\frac{1}{2\lambda} \|z-x\|^2\Big]}.\label{eq:zyx_sample}
\end{align}
Denote the solution of \eqref{eq:zyx} as $\autopar{z^*(x), y^*(x)}$ and the solution of \eqref{eq:zyx_sample} as $\autopar{z_S(x),  y_S(x)}$.
We need to bound the distance between $z^*(x)$ and $z_S(x)$, note that this $\autopar{z^*(x), y^*(x)}$ comes from a strongly-convex-strongly-concave stochastic minimax problem, where the modulus is $\frac{1-\lambda L}{\lambda}$ and $\nu$, respectively; while the other comes from the sample average approximation counterpart. By \citet[Theorem 1 and Appendix A.1]{zhang2021generalization}, 
we have the following results:
\begin{equation*}
\frac{1-\lambda L}{2\lambda}\EE \|z_S(x) - z^*(x)\|^2 + \frac{\nu}{2}\EE\|y_S(x) - y^*(x)\|^2\leq \frac{2\sqrt{2}}{n}\autopar{\frac{\hat L_x^2\lambda}{1-\lambda L}+\frac{\hat L_y^2}{\nu}},
\end{equation*}
where $\hat L_x$ is the Lipschitz continuity parameter of $f(z,y;\xi)+\frac{1}{2\lambda}\|z-x\|^2$ in $z\in\mathcal{X}$ for any given $y\in\mathcal{Y}$ and $\xi$, and $\hat L_y$ is the Lipschitz continuity parameter of $f(z,y;\xi)-\frac{\nu}{2}\|y\|^2$ in $y\in\mathcal{Y}$ for any given $z\in\mathcal{X}$ and $\xi$. 
More specifically, since $f(\cdot,\cdot;\xi)$ is $G$-Lipschitz continuous for any $\xi$, we have
\begin{equation*}
\hat L_x \leq G + \frac{2\sqrt{D_\mathcal{X}}}{\lambda},\quad
\hat L_y \leq G + \nu\sqrt{D_{\mathcal{Y}}}.
\end{equation*}
Therefore, we have
\begin{equation}
\label{eq:NCC-agnostic-expectation-bound}
\begin{split}
    \EE \|\prox_{\lambda \hat \Phi_S}(x) - \prox_{\lambda \hat \Phi}(x) \| 
    =\ & \EE \|z_S(x) - z^*(x)\|\\
    \leq\ & \sqrt{\EE \|z_S(x) - z^*(x)\|^2} 
    \leq \sqrt{\frac{2\lambda}{1-\lambda L}\frac{2\sqrt{2}}{n}\autopar{\frac{\hat L_x^2\lambda}{1-\lambda L}+\frac{ \hat L_y^2}{\nu}}}.
\end{split}
\end{equation}
Next, we show that $\|z_S(x) - z^*(x)\| - \EE\|z_S(x) - z^*(x)\| $  is a zero-mean sub-Gaussian random variable. Replacing one sample $\xi_i$ in $S$ with an i.i.d. sample $\xi_i^\prime$ and denote the new dataset as $S^{(i)}$, by \citet[Lemma 2]{zhang2021generalization}, it holds that
\begin{equation*}
\autonorm{z_S(x) - z^*(x)} - \autonorm{z_{S^{(i)}}(x) - z^*(x)} 
\leq 
\autonorm{z_S(x)  - z_{S^{(i)}}(x)}
\leq 
\frac{2}{n}\sqrt{\frac{\hat L_x^2 \lambda^2}{(1-\lambda L)^2}+\frac{\hat L_y^2 \lambda}{\nu(1-\lambda L)}},
\end{equation*}
where $z_{S^{(i)}}$ follows a similar definition of $z_S$ but with a different dataset $S^{(i)}$. By McDiarmid’s inequality (Lemma \ref{lm:McDiarmid_inequality}) and the definition of sub-Gaussian random variables, it holds that $\|z_S(x) - z^*(x)\| - \EE\|z_S(x) - z^*(x)\|$  is a zero-mean sub-Gaussian random variable with variance proxy
$
\frac{1}{n}\autopar{\frac{\hat L_x^2 \lambda^2}{(1-\lambda L)^2}+\frac{\hat L_y^2 \lambda}{\nu(1-\lambda L)}}
$. By the definition of sub-Gaussian random variable and \eqref{eq:NCC-agnostic-expectation-bound}, it holds that 
\begin{equation}
    \label{eq:NCC-proximal-operator-empirical-population}
    \begin{split}
        &
            \EE \exp\autopar{s\autonorm{\prox_{\lambda \hat \Phi_S}(x_k) - \prox_{\lambda \hat \Phi}(x_k)}} \\
        =\ &
            \EE \exp\autopar{s\automedpar{\autonorm{\prox_{\lambda \hat \Phi_S}(x_k) - \prox_{\lambda \hat \Phi}(x_k) } -\EE \autonorm{\prox_{\lambda \hat \Phi_S}(x_k) - \prox_{\lambda \hat \Phi}(x_k)}}} \\ 
            &\qquad\qquad\qquad
            \cdot\exp\autopar{s\EE \autonorm{\prox_{\lambda \hat \Phi_S}(x_k) - \prox_{\lambda \hat \Phi}(x_k)}}\\
        \leq\ &
            \EE \exp\autopar{s\automedpar{\autonorm{\prox_{\lambda \hat \Phi_S}(x_k) - \prox_{\lambda \hat \Phi}(x_k)}-\EE \autonorm{\prox_{\lambda \hat \Phi_S}(x_k) - \prox_{\lambda \hat \Phi}(x_k)}}}\\
            &\qquad\qquad\qquad
            \cdot\exp\autopar{s\sqrt{\frac{2\lambda}{1-\lambda L}\frac{2\sqrt{2}}{n}\autopar{\frac{\hat L_x^2\lambda}{1-\lambda L}+\frac{ \hat L_y^2}{\nu}}}}\\
        \leq\ &
            \exp\autopar{\frac{s^2}{2n}\autopar{\frac{\hat L_x^2 \lambda^2}{(1-\lambda L)^2}+\frac{\hat L_y^2 \lambda}{\nu(1-\lambda L)}}}
            \exp\autopar{s\sqrt{\frac{2\lambda}{1-\lambda L}\frac{2\sqrt{2}}{n}\autopar{\frac{\hat L_x^2\lambda}{1-\lambda L}+\frac{ \hat L_y^2}{\nu}}}},
    \end{split}
\end{equation}
where the second inequality uses definition of zero-mean sub-Gaussian random variable.
Combining \eqref{eq:NCC-proximal-operator-empirical-population} with \eqref{eq:NCC-agnostic-key}, for
\begin{equation}
    \lambda = \frac{1}{2L},\quad
    \upsilon = \frac{\epsilon\lambda(1-\lambda L)}{8} =\frac{\epsilon}{32L},\quad
    s = \sqrt{2n\log(Q)\autopar{\frac{\hat L_x^2 \lambda^2}{(1-\lambda L)^2}+\frac{\hat L_y^2 \lambda}{\nu(1-\lambda L)}}^{-1}},
\end{equation}
it holds that
\begin{equation}
\label{eq:NCC-agnostic-final}
\begin{aligned}
&
    \EE\max_{x\in\mathcal{X}}\|\nabla  \Phi_S^\lambda(x)-\nabla \Phi^\lambda(x)\| \\
\leq\ &
    2\sqrt{\frac{\nu D_\mathcal{Y}}{\lambda(1-\lambda(L+\nu))}} + \frac{2\upsilon}{\lambda (1-\lambda(L+\nu))} \\
    & + \frac{1}{\lambda s}\log\autopar{Q\exp\autopar{\frac{s^2}{2n}\autopar{\frac{\hat L_x^2 \lambda^2}{(1-\lambda L)^2}+\frac{\hat L_y^2 \lambda}{\nu(1-\lambda L)}}}}\\
    & \qquad\qquad\qquad\qquad
    +\frac{1}{\lambda s}\log\autopar{\exp\autopar{s\sqrt{\frac{2\lambda}{1-\lambda L}\frac{2\sqrt{2}}{n}\autopar{\frac{\hat L_x^2\lambda}{1-\lambda L}+\frac{ \hat L_y^2}{\nu}}}}}\\
\leq\ &
    2\sqrt{\frac{\nu D_\mathcal{Y}}{\lambda(1-\lambda L)}} + \frac{1}{\lambda s}\log(Q)+\frac{1}{\lambda s}\frac{s^2}{2n}\autopar{\frac{\hat L_x^2 \lambda^2}{(1-\lambda L)^2}+\frac{\hat L_y^2 \lambda}{\nu(1-\lambda L)}}\\
    &\qquad\qquad\qquad
    +\frac{1}{\lambda s}
    s\sqrt{\frac{2\lambda}{1-\lambda L}\frac{2\sqrt{2}}{n}\Big(\frac{\hat L_x^2\lambda}{1-\lambda L}+\frac{ \hat L_y^2}{\nu}\Big)} + \frac{2\upsilon}{\lambda (1-\lambda L)}\\
=\ &
    2\sqrt{\frac{\nu D_\mathcal{Y}}{\lambda(1-\lambda L)}} + \frac{\log(Q)}{\lambda s}+\frac{1}{\lambda }\frac{s}{2n}\autopar{\frac{\hat L_x^2 \lambda^2}{(1-\lambda L)^2}+\frac{\hat L_y^2 \lambda}{\nu(1-\lambda L)}}\\
    &\ +\frac{1}{\lambda}
    \sqrt{\frac{2\lambda}{1-\lambda L}\frac{2\sqrt{2}}{n}\autopar{\frac{\hat L_x^2\lambda}{1-\lambda L}+\frac{ \hat L_y^2}{\nu}}}+ \frac{\epsilon}{4}\\
=\ &
    2\sqrt{4 L\nu D_\mathcal{Y}} + 4L\sqrt{\frac{\log(Q)}{2n}\autopar{\frac{\hat L_x^2}{L^2} +\frac{ \hat L_y^2 }{\nu L}}} + 2L\sqrt{\frac{4\sqrt{2}}{L n}\autopar{\frac{\hat L_x^2}{L}+\frac{ \hat L_y^2}{\nu}}}+\frac{\eps}{4}\\
=\ &
    2\sqrt{4 L\nu D_\mathcal{Y}} + 4L\sqrt{\frac{\log(Q)}{2n}\autopar{\frac{\hat L_x^2}{L^2} +\frac{ \hat L_y^2 }{\nu L}}}\\
    &\qquad\qquad\qquad
    + 2L\sqrt{\frac{4\sqrt{2}}{L n}\autopar{\frac{( G+ 4L\sqrt{D_\mathcal{X}})^2}{L}+\frac{ ( G+\nu\sqrt{D_\mathcal{Y}})^2}{\nu}}}+\frac{\eps}{4}.
\end{aligned}  
\end{equation}
Here 
the first equality holds by the selection of $\upsilon$, the second equality holds by the selection of $\lambda$ and $s$, and the last equality holds by plugging in $\hat L_x$ and $\hat L_y$. Note that $\upsilon$, $s$, and $\nu$ are only used for analysis purposes, and $\lambda$ is only used in the definition of gradient mapping. Thus one has free choices on these parameters. 
Since $Q = \cO\autopar{\autopar{\frac{D_\mathcal{X}}{\upsilon}}^d}$, then we choose $\upsilon=\tilde\cO\autopar{\sqrt{\frac{d}{n}}}$ in the right-hand side above, which verifies the first statement. For the sample complexity result, to make sure that the right-hand side of \eqref{eq:NCC-agnostic-final} of order  $\cO(\eps)$, it suffices to have 
\begin{equation}
    \nu=\cO(\eps^2),\quad
    n = \cO\autopar{\frac{\log(Q)}{\nu}\eps^{-2}}= \cO\autopar{d\eps^{-4}\log\autopar{\eps^{-1}}},
\end{equation}
which concludes the proof.
\end{proof}

\section{Proof of Theorem \ref{thm:stability and generalization}}
\label{apdx:NCSC stability gen}
For simplicity we define the following notations:
\begin{equation}
    \begin{split}
        F(x,y)\triangleq &\E_\xi\automedpar{f(x,y;\xi)},\quad
        \Phi(x)\triangleq \max_y F(x,y),\quad
        F_S(x,y)\triangleq \frac{1}{n}\sum_{i=1}^n\automedpar{f(x,y;\xi_i)},\quad
        \Phi_S(x)\triangleq \max_y F_S(x,y),\\
        y^*(x)\triangleq & \argmax_y F(x,y),\quad
        y_S^*(x)\triangleq \argmax_y F_S(x,y),\quad
        \Phi(x;\xi)\triangleq \max_y f(x,y;\xi),
    \end{split}
\end{equation}
and the Moreau envelope of a function $\Phi$:
\begin{equation}
    \label{eq:Moreau_Envelope}
    \Phi^\lambda(x)\triangleq\min_{z\in\mathcal{X}}\autobigpar{\Phi(x)+\frac{1}{2\lambda}\autonorm{z-x}_2^2},\quad
    \prox_{\lambda \Phi}(x)\triangleq\argmin_{z\in\mathcal{X}}\autobigpar{\Phi(x)+\frac{1}{2\lambda}\autonorm{z-x}_2^2},
\end{equation}
similar notations can be defined for $\Phi_S$, which we do not repeat here.

\begin{definition}[Uniform Stability]
    
    We say a randomized algorithm $\mathcal{A}$ is $\delta$-uniformly stable in $x$-gradients if for every two dataset $S, S'$ which differ in only one sample, for every $\xi\in\Xi$ we have
    \begin{equation}
        \sup_\xi\mathbb{E}_\mathcal{A}\autonorm{\nabla_x f\autopar{\mathcal{A}_x(S),\mathcal{A}_y(S);\xi}-\nabla_x f\autopar{\mathcal{A}_x(S'), \mathcal{A}_y(S');\xi}}^2\leq\delta^2.
    \end{equation}
    
\end{definition}

\begin{lemma}[Concentration of Optimizers]
    \label{lm:distance_maximizer_y_yS}
    For $y^*$ and $y_S^*$ defined above, with Assumption \ref{as:ncsc-ncc}, we have for any $x\in\mathcal{X}$,
    \begin{equation}
        \autonorm{y^*(x)-y_S^*(x)}\leq\frac{1}{\mu}\autonorm{\nabla_y F_S(x,y^*(x))-\nabla_y F(x,y^*(x))}.
    \end{equation}
\end{lemma}

\begin{proof}
    By the optimality of $y^*(x)$ and $y_S^*(x)$, we have for any $y\in\mathcal{Y}$
    \begin{equation}
        \begin{split}
            \autoprod{y-y^*(x), \nabla_y F(x,y^*(x))}&\leq 0\\
            \autoprod{y-y_S^*(x), \nabla_y F_S(x,y_S^*(x))}&\leq 0.
        \end{split}
    \end{equation}
Setting $y=y_S^*(x)$ and $y=y^*(x)$ in the above inequalities respectively, we have
    \begin{equation}
        \label{eq:concentration_maximizer_1}
        \autoprod{y_S^*(x)-y^*(x), \nabla_y F(x,y^*(x))-\nabla_y F_S(x,y_S^*(x))}\leq 0.
    \end{equation}
In addition, by strong concavity of $F_S(x,\cdot)$, we have
    \begin{equation}
        \label{eq:concentration_maximizer_2}
        \autoprod{y_S^*(x)-y^*(x), \nabla_y F_S(x,y_S^*(x))-\nabla_y F_S(x,y^*(x))}+\mu\autonorm{y_S^*(x)-y^*(x)}^2\leq 0.
    \end{equation}
Combining \eqref{eq:concentration_maximizer_1} and \eqref{eq:concentration_maximizer_2}, we have
    \begin{equation}
        \label{eq:concentration_maximizer_3}
        \autoprod{y_S^*(x)-y^*(x), \nabla_y F(x,y^*(x))-\nabla_y F_S(x,y^*(x))}+\mu\autonorm{y_S^*(x)-y^*(x)}^2\leq 0.
    \end{equation}
Rearranging terms, it holds that
    \begin{equation}
        \begin{split}
            \mu\autonorm{y_S^*(x)-y^*(x)}^2
            \leq\ &
            \autoprod{y_S^*(x)-y^*(x), \nabla_y F_S(x,y^*(x))-\nabla_y F(x,y^*(x))}\\
            \leq\ &
            \autonorm{y_S^*(x)-y^*(x)}\cdot \autonorm{\nabla_y F_S(x,y^*(x))-\nabla_y F(x,y^*(x))},
        \end{split}
    \end{equation}
    which implies
    \begin{equation}
        \begin{split}
            \autonorm{y_S^*(x)-y^*(x)}
            \leq\ &
            \frac{1}{\mu}\autonorm{\nabla_y F_S(x,y^*(x))-\nabla_y F(x,y^*(x))}.
        \end{split}
    \end{equation}
It concludes the proof.
\end{proof}

\begin{lemma}[Stability of Optimizers]
    \label{lm:distance_maximizer_yS1_yS2}
    For $y_S^*$ and $y_{S'}^*$ defined above where $S$ and $S'$ are two dataset differing in only one sample ($\xi_i$ and $\xi_i'$), with Assumption \ref{as:ncsc-ncc} while $\mu>0$, we have for any $x\in\mathcal{X}$,
    \begin{equation}
        \autonorm{y_S^*(x)-y_{S'}^*(x)}\leq\frac{1}{\mu}\autonorm{\nabla_y F_S(x,y_{S'}^*(x))-\nabla_y F_{S'}(x,y_{S'}^*(x))}\leq\frac{2G}{n\mu}.
    \end{equation}
\end{lemma}

\begin{proof}
    The proof is similar to that of Lemma \ref{lm:distance_maximizer_y_yS}. By the optimality of $y_S^*(x)$ and $y_{S'}^*(x)$, we have for any $y\in\mathcal{Y}$
    \begin{equation}
        \begin{split}
            \autoprod{y-y_S^*(x), \nabla_y F_S(x,y_S^*(x))}&\leq 0\\
            \autoprod{y-y_{S'}^*(x), \nabla_y F_{S'}(x,y_{S'}^*(x))}&\leq 0.
        \end{split}
    \end{equation}
    Setting $y=y_{S'}^*(x)$ and $y=y_S^*(x)$ in the above inequalities respectively, we have
    \begin{equation}
        \label{eq:stability_maximizer_1}
        \autoprod{y_S^*(x)-y_{S'}^*(x), \nabla_y F_{S'}(x,y_{S'}^*(x))-\nabla_y F_S(x,y_S^*(x))}\leq 0.
    \end{equation}
    In addition, by strong concavity of $F_S(x,\cdot)$, we have
    \begin{equation}
        \label{eq:stability_maximizer_2}
        \autoprod{y_S^*(x)-y_{S'}^*(x), \nabla_y F_S(x,y_S^*(x))-\nabla_y F_S(x,y_{S'}^*(x))}+\mu\autonorm{y_S^*(x)-y_{S'}^*(x)}^2\leq 0.
    \end{equation}
    Combining \eqref{eq:stability_maximizer_1} and \eqref{eq:stability_maximizer_2}, we have
    \begin{equation}
        \label{eq:stability_maximizer_3}
        \autoprod{y_S^*(x)-y_{S'}^*(x), \nabla_y F_S(x,y_{S'}^*(x))-\nabla_y F_S(x,y_{S'}^*(x))}+\mu\autonorm{y_S^*(x)-y_{S'}^*(x)}^2\leq 0.
    \end{equation}
    Rearranging terms, it holds that
    \begin{equation}
        \begin{split}
            \mu\autonorm{y_S^*(x)-y_{S'}^*(x)}^2
            \leq\ &
            \autoprod{y_S^*(x)-y_{S'}^*(x), \nabla_y F_S(x,y_{S'}^*(x))-\nabla_y F_{S'}(x,y_{S'}^*(x))}\\
            \leq\ &
            \autonorm{y_S^*(x)-y_{S'}^*(x)}\cdot \autonorm{\nabla_y F_S(x,y_{S'}^*(x))-\nabla_y F_{S'}(x,y_{S'}^*(x))},
        \end{split}
    \end{equation}
    which implies
    \begin{equation}
        \begin{split}
            \autonorm{y_S^*(x)-y_{S'}^*(x)}
            \leq\ &
            \frac{1}{\mu}\autonorm{\nabla_y F_S(x,y_{S'}^*(x))-\nabla_y F_{S'}(x,y_{S'}^*(x))}\\
            =\ &
            \frac{1}{\mu}\autonorm{\frac{1}{n}\autopar{\nabla_y f(x,y_{S'}^*(x);\xi_i)-\nabla_y f(x,y_{S'}^*(x);\xi_i')}}
            \leq
            \frac{2G}{n\mu},
        \end{split}
    \end{equation}
    which concludes the proof. Here the equality above is due to the variables being the same ($x,y_{S'}^*(x)$), while $S$ and $S'$ differ in only one sample.
\end{proof}

\begin{theorem}[Stability and Generalization, NC-SC]
    Let $\mathcal{A}$ be an $\delta$-uniformly primal stable algorithm, for any function $f$ satisfying Assumption \ref{as:ncsc-ncc} with $\mu>0$, we have
    \begin{equation}
        \mathbb{E}_{\mathcal{A},S}\autonorm{\nabla\Phi\autopar{\mathcal{A}_x(S)}-\nabla\Phi_S\autopar{\mathcal{A}_x(S)}}
        \leq
        (1+\kappa)\autopar{4\delta+\frac{ G}{\sqrt{n}}}.
    \end{equation}
\end{theorem}

\begin{proof}
    Following the definition, we have
    \begin{equation}
        \begin{split}
            &\nabla\Phi\autopar{\mathcal{A}_x(S)}-\nabla\Phi_S\autopar{\mathcal{A}_x(S)}\\
            =\ &
            \nabla_x F\autopar{\mathcal{A}_x(S), y^*(\mathcal{A}_x(S))}
            -
            \nabla_x F_S\autopar{\mathcal{A}_x(S), y_S^*(\mathcal{A}_x(S))}\\
            =\ &
            \nabla_x F\autopar{\mathcal{A}_x(S), y^*(\mathcal{A}_x(S))}
            -
            \nabla_x F_S\autopar{\mathcal{A}_x(S), y^*(\mathcal{A}_x(S))}
            +
            \nabla_x F_S\autopar{\mathcal{A}_x(S), y^*(\mathcal{A}_x(S))}
            -
            \nabla_x F_S\autopar{\mathcal{A}_x(S), y_S^*(\mathcal{A}_x(S))},
        \end{split}
    \end{equation}
    so we know that
    \begin{equation}
        \begin{split}
            &\autonorm{\nabla\Phi\autopar{\mathcal{A}_x(S)}-\nabla\Phi_S\autopar{\mathcal{A}_x(S)}}\\
            \leq\ &
            \autonorm{
            \nabla_x F\autopar{\mathcal{A}_x(S), y^*(\mathcal{A}_x(S))}
            -
            \nabla_x F_S\autopar{\mathcal{A}_x(S), y^*(\mathcal{A}_x(S))}
            }\\
            &\qquad\qquad\qquad +
            \autonorm{
            \nabla_x F_S\autopar{\mathcal{A}_x(S), y^*(\mathcal{A}_x(S))}
            -
            \nabla_x F_S\autopar{\mathcal{A}_x(S), y_S^*(\mathcal{A}_x(S))}
            },
        \end{split}
    \end{equation}
    for the first term above, by \citet[Theorem 2]{lei2022stability} (i.e., regarding $\autopar{\mathcal{A}_x(S), y^*(\mathcal{A}_x(S))}$ as one single variable to recover their conclusion), we have
    \begin{equation}
        \label{eq:stability_term1}
        \mathbb{E}_{\mathcal{A},S}\autonorm{
        \nabla_x F\autopar{\mathcal{A}_x(S), y^*(\mathcal{A}_x(S))}
        -
        \nabla_x F_S\autopar{\mathcal{A}_x(S), y^*(\mathcal{A}_x(S))}
        }
        \leq
        4\delta+\sqrt{\frac{\text{Var}(\nabla_x f)}{n}}
        \leq
        4\delta+\frac{G}{\sqrt{n}},
    \end{equation}
    for the second term above, by Lemma \ref{lm:distance_maximizer_y_yS}, we have
    \begin{equation}
        \begin{split}
            &
            \EE_{\mathcal{A},S}\autonorm{
            \nabla_x F_S\autopar{\mathcal{A}_x(S), y^*(\mathcal{A}_x(S))}
            -
            \nabla_x F_S\autopar{\mathcal{A}_x(S), y_S^*(\mathcal{A}_x(S))}
            }\\
            \leq\ &
            L\EE_{\mathcal{A},S}\autonorm{y^*(\mathcal{A}_x(S))-y_S^*(\mathcal{A}_x(S))}\\
            \leq\ &
            \kappa\EE_{\mathcal{A},S}\autonorm{\nabla_y F_S(\mathcal{A}_x(S),y^*(\mathcal{A}_x(S)))-\nabla_y F(\mathcal{A}_x(S),y^*(\mathcal{A}_x(S)))}\\
            \leq\ &
            \kappa\autopar{4\delta+\sqrt{\frac{\text{Var}(\nabla_y f)}{n}}}\\
            \leq\ &
            \kappa\autopar{4\delta+\frac{G}{\sqrt{n}}},
        \end{split}
    \end{equation}
    where the third inequality applies the same argument as that in \eqref{eq:stability_term1}. We conclude the proof by combining the two bounds above together. 
\end{proof}

\section{Proof of Theorem \ref{thm:stability and generalization NCC}}
\label{apdx:NCC stability gen}
The proof uses the idea from \citet[Theorem 3]{lei2022stability} and our proof of Theorem \ref{thm:NC-C Agnostic}. Unlike \citet{lei2022stability} which considers the minimization case, with $\Phi(x)\neq\EE[\Phi(x;\xi)]$, we need some modification in the proof. To address the non-uniqueness of $y^*(x)$ in the NC-C case, similar to the uniform convergence analysis in the NC-C case (Theorem \ref{thm:NC-C Agnostic}), we resort to the regularized objective in the proof to characterize corresponding distances.

For convenience, we recall the definition of regularized objective functions here.
\begin{equation}
    \label{eq:regularized_obj}
    \begin{split}
        \widehat{\Phi}(x)=\max_{y\in\mathcal{Y}} F(x,y)-\frac{\nu}{2}\autonorm{y}^2&,\quad
        \widehat{\Phi}_S(x)=\max_{y\in\mathcal{Y}} F_S(x,y)-\frac{\nu}{2}\autonorm{y}^2,\\
        \widehat{y}^*(x)=\argmax_{y\in\mathcal{Y}} F(x,y)-\frac{\nu}{2}\autonorm{y}^2&,\quad
        \widehat{y}_S^*(x)=\argmax_{y\in\mathcal{Y}} F_S(x,y)-\frac{\nu}{2}\autonorm{y}^2.
    \end{split}
\end{equation}
In addition, following the notation in \cite{lei2022stability}, we define
\begin{equation}
    \label{eq:proximal_operator_regularized_obj}
    \begin{split}
        \widetilde{w}_S&=\prox_{\frac{\widehat{\Phi}}{2L}}(\mathcal{A}_x(S))=\argmin_{x\in\mathcal{X}}\autobigpar{\widehat{\Phi}(x)+L\autonorm{x-\mathcal{A}_x(S)}^2},\\
        w_S&=\prox_{\frac{\widehat{\Phi}_S}{2L}}(\mathcal{A}_x(S))=
    \argmin_{x\in\mathcal{X}}\autobigpar{\widehat{\Phi}_S(x)+L\autonorm{x-\mathcal{A}_x(S)}^2}.
    \end{split}
\end{equation}
As discussed in Appendix \ref{apdx_sec:proof_NCC_unif_convergence}, the function $F(x,y)-\frac{\nu}{2}\autonorm{y}^2$ is $(L+\nu)$-smooth, and the function $\widehat{\Phi}(x)$ is $(L+\nu)$-weakly-convex (the same hold for $F_S(x,y)-\frac{\nu}{2}\autonorm{y}^2$ and $\widehat{\Phi}_S(x)$). 

First, we build up a connection between algorithm stability and proximal operators to facilitate the analysis.
\begin{lemma}[Algorithm Stability and Proximal Operators]
    \label{lm:stability_NCC_closeness}
    Let $\mathcal{A}$ be an algorithm. For any function $f$ satisfying Assumption \ref{as:ncsc-ncc} with $\mu=0$, we have the following inequalities for any two neighboring dataset $S$ and $S'$, we have
    \begin{equation}
        \begin{split}
            \autonorm{\widetilde{w}_S-\widetilde{w}_{S'}}\leq\ &
            \frac{2L}{L-\nu}\autonorm{\mathcal{A}(S)-\mathcal{A}(S')}\\
            \autonorm{w_S-w_{S'}}\leq\ &
            \frac{2L}{L-\nu}\autonorm{\mathcal{A}(S)-\mathcal{A}(S')}+\frac{2G}{n(L-\nu)}+\frac{2L(G+\nu \sqrt{D_\mathcal{Y}})}{n\nu(L-\nu)},
        \end{split}
    \end{equation}
    where $\widehat{\Phi}$ and $\widehat{\Phi}_S$ follows the definitions in \eqref{eq:regularized_obj} and \eqref{eq:proximal_operator_regularized_obj}.
\end{lemma}

The proof basically follows the proof of \citet[Lemma 15 and 16]{lei2022stability} with some differences in detailed parameters.

\begin{proof}
     For the first result, note that $\widehat{\Phi}(x)$ is $(L+\nu)$-weakly-convex and differentiable, so we have
    \begin{equation}
        \autoprod{\widetilde{w}_S-\widetilde{w}_{S'}, \nabla\widehat{\Phi}(\widetilde{w}_S)-\nabla\widehat{\Phi}(\widetilde{w}_{S'})}\geq -(L+\nu)\autonorm{\widetilde{w}_S-\widetilde{w}_{S'}}^2.
        \label{eq:app-wcphi}
    \end{equation}
    On the other hand, by the optimality of $\widetilde{w}_S$, we have
    \begin{equation}
        -2L\autopar{\widetilde{w}_S-\mathcal{A}_x(S)} -
        \nabla\widehat{\Phi}\widetilde{w}_S) \in\partial\mathcal{I}_\mathcal{X}(\widetilde{w}_S),
        \quad
        -2L\autopar{\widetilde{w}_{S'}-\mathcal{A}_x(S')} -\nabla\widehat{\Phi}(\widetilde{w}_{S'}) \in\partial\mathcal{I}_\mathcal{X}(\widetilde{w}_{S'}),
    \end{equation}
    where $\mathcal{I}_\mathcal{X}(x)$ is the indicator function of the set $\mathcal{X}$, i.e., $\mathcal{I}_\mathcal{X}(x)=0$ if $x\in\mathcal{X}$ and $\mathcal{I}_\mathcal{X}(x)=\infty$ otherwise. Since $\mathcal{X}$ is convex, the subgradient $\partial\mathcal{I}_\mathcal{X}$ is monotone, and thus
    \begin{equation}
        \begin{split}
            \autoprod{
                \widetilde{w}_S-\widetilde{w}_{S'},
                2L\autopar{\widetilde{w}_{S'}-\mathcal{A}_x(S')} - 2L\autopar{\widetilde{w}_S-\mathcal{A}_x(S)} +
                \nabla\widehat{\Phi}(\widetilde{w}_{S'}) -
                \nabla\widehat{\Phi}(\widetilde{w}_S)
            } & =
            \autoprod{
                \widetilde{w}_S-\widetilde{w}_{S'},
                \partial\mathcal{I}_\mathcal{X}(\widetilde{w}_S) -
                \partial\mathcal{I}_\mathcal{X}(\widetilde{w}_{S'})
            } \\ & \geq 0.
        \end{split}
        \label{eq:app-optimalphi}
    \end{equation}
    Combining \eqref{eq:app-wcphi} and \eqref{eq:app-optimalphi}, it follows that
    \begin{equation}
        \autoprod{\widetilde{w}_S-\widetilde{w}_{S'}, 2L\autopar{\widetilde{w}_{S'}-\mathcal{A}_x(S')}-2L\autopar{\widetilde{w}_S-\mathcal{A}_x(S)}}\geq -(L+\nu)\autonorm{\widetilde{w}_S-\widetilde{w}_{S'}}^2.
    \end{equation}
    Rearranging the terms, we have
    \begin{equation}
        (L-\nu)\autonorm{\widetilde{w}_S-\widetilde{w}_{S'}}^2
        \leq
        2L\autoprod{\widetilde{w}_S-\widetilde{w}_{S'},\mathcal{A}_x(S)-\mathcal{A}_x(S')}
        \leq
        2L\autonorm{\widetilde{w}_S-\widetilde{w}_{S'}}\autonorm{\mathcal{A}_x(S)-\mathcal{A}_x(S')}.
    \end{equation}
    We obtain the first result by dividing both sides by $(L-\nu)\autonorm{\widetilde{w}_S-\widetilde{w}_{S'}}$.

    For the second statement, applying the fact that $\widehat\Phi_S$ is weakly-convex and differentiable, 
    \begin{equation}
        \autoprod{w_S-w_{S'}, \nabla\widehat{\Phi}_S(w_S)-\nabla\widehat{\Phi}_S(w_{S'})}\geq -(L+\nu)\autonorm{w_S-w_{S'}}^2.
    \end{equation}
    Similar as \eqref{eq:app-optimalphi}, by the optimality condition of $w_S$ and $w_{S'}$,
    \begin{equation}
        \autoprod{
            w_S-w_{S'},
            2L\autopar{w_{S'}-\mathcal{A}_x(S')} -
            2L\autopar{w_S-\mathcal{A}_x(S)} +
            \nabla\widehat{\Phi}_{S'}(w_{S'}) -
            \nabla\widehat{\Phi}_S(w_S)
        } \geq 0.
    \end{equation}
    Therefore, by the above two equations, we obtain that
    \begin{equation}
        -(L+\nu)\autonorm{w_S-w_{S'}}^2 \leq
        \autoprod{
            w_S-w_{S'},
            2L\autopar{w_{S'}-\mathcal{A}_x(S')} -
            2L\autopar{w_S-\mathcal{A}_x(S)} +
            \nabla\widehat{\Phi}_{S'}(w_{S'}) -
            \nabla\widehat{\Phi}_S(w_{S'})
        }.
    \end{equation}
    By the definition of $\widehat{\Phi}_S$ and $w_S$, we rewrite the additional term $\nabla\widehat{\Phi}_{S'}(w_{S'}) -\nabla\widehat{\Phi}_S(w_{S'})$ as
    \begin{equation}
        \begin{split}
            &
            \nabla\widehat{\Phi}_{S'}(w_{S'}) -
            \nabla\widehat{\Phi}_S(w_{S'}) \\
            = &         
            \nabla_x\autopar{
                F_{S'}(w_{S'}; \widehat{y}_{S'}^*(w_{S'})) -
                \frac{\nu}{2}\autonorm{\widehat{y}_{S'}^*(w_{S'})}^2
            } -
            \nabla\widehat{\Phi}_S(w_{S'}) \\
            = &
            \nabla_x F_{S'}(w_{S'}; \widehat{y}_{S'}^*(w_{S'})) -
            \nabla\widehat{\Phi}_S(w_{S'}) \\
            = &
            \nabla_x F_{S'}(w_{S'}; \widehat{y}_{S'}^*(w_{S'})) -
            \nabla_x F_S(w_{S'}; \widehat{y}_{S'}^*(w_{S'})) +
            \nabla_x F_S(w_{S'}; \widehat{y}_{S'}^*(w_{S'})) -
            \nabla_x F_S(w_{S'}; \widehat{y}_S^*(w_{S'})) \\
            = &
            \underbrace{\frac{1}{n}\nabla_x f(w_{S'}, \widehat{y}_{S'}^*(w_{S'});\xi_i')-\frac{1}{n}\nabla_x f(w_{S'}, \widehat{y}_{S'}^*(w_{S'});\xi_i)}_{E_1} +
            \underbrace{\nabla_x F_S(w_{S'}; \widehat{y}_{S'}^*(w_{S'}))-\nabla_x F_S(w_{S'}; \widehat{y}_S^*(w_{S'}))}_{E_2},
        \end{split}
    \end{equation}
    where the third equation holds since $\nabla\widehat{\Phi}_S(w_{S'})=\nabla_x F_S(w_{S'}; \widehat{y}_S^*(w_{S'}))$.
    Thus it holds that
    \begin{equation}
        \begin{split}
            -(L+\nu)\autonorm{w_S-w_{S'}}^2
            \leq
            \autoprod{w_S-w_{S'}, -2L\autopar{w_S-\mathcal{A}_x(S)}+2L\autopar{w_{S'}-\mathcal{A}_x(S')}+E_1+E_2}.
        \end{split}
    \end{equation}
    Rearranging terms, we have
    \begin{equation}
    \label{eq:w_S}
        \begin{split}
            &(L-\nu)\autonorm{w_S-w_{S'}}^2\\
            \leq\ &
            \autoprod{w_S-w_{S'}, 2L\autopar{\mathcal{A}_x(S)-\mathcal{A}_x(S')}+E_1+E_2}\\
            \leq\ &
            \autonorm{w_S-w_{S'}} \autonorm{2L\autopar{\mathcal{A}_x(S)-\mathcal{A}_x(S')}+E_1+E_2}\\
            \leq\ &
            \autonorm{w_S-w_{S'}}\autopar{2L\autonorm{\autopar{\mathcal{A}_x(S)-\mathcal{A}_x(S')}}+\autonorm{E_1}+\autonorm{E_2}}\\
            \leq\ &
            \autonorm{w_S-w_{S'}}\autopar{2L\autonorm{\autopar{\mathcal{A}_x(S)-\mathcal{A}_x(S')}}+\frac{2G}{n}+\frac{2L(G+\nu \sqrt{D_\mathcal{Y}})}{n\nu}},
        \end{split}
    \end{equation}
    where the last inequality uses the fact that that $\autonorm{E_1}\leq2G/n$ via  Lipschitz continuity, and 
    $$
    \autonorm{E_2}\leq L\autonorm{\widehat{y}_{S'}^*(w_{S'})-\widehat{y}_S^*(w_{S'})}\stackrel{\text{Lemma \ref{lm:distance_maximizer_yS1_yS2}}}{\leq}\frac{2L(G+\nu \sqrt{D_\mathcal{Y}})}{n\nu}.
    $$ 
    It concludes the proof by diving $(L-\nu)\autonorm{w_S-w_{S'}}$ on both sides of \eqref{eq:w_S}.
\end{proof}

\begin{lemma}
    \label{lm:NCC_Stability_term1}
    Let $\mathcal{A}$ be an $\delta$-uniformly primal argument stable algorithm. For any function $f$ satisfying Assumption \ref{as:ncsc-ncc} with $\mu=0$, we have
    \begin{equation}
        \EE\automedpar{\widehat{\Phi}_S(\widetilde{w}_S)-\widehat{\Phi}(\widetilde{w}_S)}
        \leq
        \frac{2GL(L+2\nu)}{\nu(L-\nu)}\delta+\frac{G}{\nu}\autopar{4\sqrt{\frac{8L^4(L+2\nu)^2}{\nu^2(L-\nu)^2}}\delta+\frac{G}{\sqrt{n}}}+\frac{\nu}{2}D_{\mathcal{Y}}.
    \end{equation}
\end{lemma}

\begin{proof}
    Note that 
    \begin{equation}
        \begin{split}
            &\EE\automedpar{\widehat{\Phi}_S(\widetilde{w}_S)-\widehat{\Phi}(\widetilde{w}_S)}\\
            =\ &
            \EE\automedpar{F_S\autopar{\widetilde{w}_S,\widehat{y}_S^*(\widetilde{w}_S)}-F\autopar{\widetilde{w}_S,\widehat{y}^*(\widetilde{w}_S)}-\frac{\nu}{2}\autonorm{\widehat{y}_S^*(\widetilde{w}_S)}^2+\frac{\nu}{2}\autonorm{\widehat{y}^*(\widetilde{w}_S)}^2}\\
            \leq\ &
            \EE\automedpar{
            \underbrace{F_S\autopar{\widetilde{w}_S,\widehat{y}_S^*(\widetilde{w}_S)}-F_S\autopar{\widetilde{w}_S,\widehat{y}^*(\widetilde{w}_S)}}_{H_1}
            +
            \underbrace{F_S\autopar{\widetilde{w}_S,\widehat{y}^*(\widetilde{w}_S)}-F\autopar{\widetilde{w}_S,\widehat{y}^*(\widetilde{w}_S)}}_{H_2}
            }+\frac{\nu}{2}D_\mathcal{Y}.
        \end{split}
    \end{equation}
    We bound $H_2$ via the stability argument of $f(\widetilde{w}_S,\widehat{y}^*(\widetilde{w}_S);\xi)$, i.e., regarding $\autopar{\widetilde{w}_S,\widehat{y}_S^*(\widetilde{w}_S)}$ as one single variable.
    \begin{equation}
        \label{eq:stability_value_composed_alg}
        \begin{split}
            &\EE\automedpar{f(\widetilde{w}_S,\widehat{y}^*(\widetilde{w}_S);\xi)}
            -
            \EE\automedpar{f(\widetilde{w}_{S'},\widehat{y}^*(\widetilde{w}_{S'});\xi)}\\
            \leq\ &
            G\EE\automedpar{\autonorm{\widetilde{w}_S-\widetilde{w}_{S'}}+\autonorm{\widehat{y}^*(\widetilde{w}_S)-\widehat{y}^*(\widetilde{w}_{S'})}}\\
            \leq\ &
            G\EE\automedpar{\autonorm{\widetilde{w}_S-\widetilde{w}_{S'}}+\frac{L+\nu}{\nu}\autonorm{\widetilde{w}_S-\widetilde{w}_{S'}}}\\
            \leq\ &
            G\EE\automedpar{\autopar{1+\frac{L+\nu}{\nu}}\cdot \frac{2L}{L-\nu}\autonorm{\mathcal{A}_x(S)-\mathcal{A}_x(S')}}\\
            \leq\ &
            \frac{L+2\nu}{\nu}\cdot\frac{2GL}{L-\nu}\EE\automedpar{\autonorm{\mathcal{A}_x(S)-\mathcal{A}_x(S')}}\\
            \leq\ &
            \frac{2GL(L+2\nu)}{\nu(L-\nu)}\delta,
        \end{split}
    \end{equation}
    where the second inequality uses \citet[Lemma 4.3]{lin2020gradient}, and the fact that $\hat y^*$ is the optimal solution of a $(L+\nu)$-smooth and $\nu$-strongly concave maximization problem defined in \eqref{eq:regularized_obj}; the third inequality is due to Lemma \ref{lm:stability_NCC_closeness}, and the last inequality follows the definition of $\delta$-uniform primal argument stability. So we have the ``composed algorithm" $\widetilde{w}_S$ is stable\footnote{Here we call the iteration $\widetilde{w}_S=\prox_{\frac{\widehat{\Phi}}{2L}}(\mathcal{A}_x(S))=\argmin_{x\in\mathcal{X}}\autobigpar{\widehat{\Phi}(x)+L\autonorm{x-\mathcal{A}_x(S)}^2}$ as an algorithm regarding that it is a composition of the algorithm $\mathcal{A}$ and the proximal operator.} in function values, which implies~\citep{hardt2016train}
    \begin{equation}
        \EE\automedpar{F_S\autopar{\widetilde{w}_S,\widehat{y}^*(\widetilde{w}_S)}-F\autopar{\widetilde{w}_S,\widehat{y}^*(\widetilde{w}_S)}}\leq
        \frac{2GL(L+2\nu)}{\nu(L-\nu)}\delta.
    \end{equation}
    
    For the term $H_1$ above, we have
    \begin{equation}
        \label{eq:H_1_analysis}
        \begin{split}
            &\EE\automedpar{F_S\autopar{\widetilde{w}_S,\widehat{y}_S^*(\widetilde{w}_S)}-F_S\autopar{\widetilde{w}_S,\widehat{y}^*(\widetilde{w}_S)}}\\
            \leq\ &
            G\EE\autonorm{\widehat{y}_S^*(\widetilde{w}_S)-\widehat{y}^*(\widetilde{w}_S)}\\
            \leq\ &
            \frac{G}{\nu}\EE\autonorm{\nabla_yF_S\autopar{\widetilde{w}_S, \widehat{y}^*(\widetilde{w}_S)}-\nu\widehat{y}^*(\widetilde{w}_S)-\nabla_yF\autopar{\widetilde{w}_S, \widehat{y}^*(\widetilde{w}_S)}+\nu\widehat{y}^*(\widetilde{w}_S)}\\
            =\ &
            \frac{G}{\nu}\EE\autonorm{\nabla_yF_S\autopar{\widetilde{w}_S, \widehat{y}^*(\widetilde{w}_S)}-\nabla_yF\autopar{\widetilde{w}_S, \widehat{y}^*(\widetilde{w}_S)}},
        \end{split}
    \end{equation}
    where the second inequality applies Lemma \ref{lm:distance_maximizer_y_yS}. We further upper bound the RHS above using the stability argument. For $\nabla_y f(\widetilde{w}_S, \widehat{y}^*(\widetilde{w}_S);\xi)$, similar to the same argument as in \eqref{eq:stability_value_composed_alg}, we have 
    \begin{equation}
        \begin{split}
            &\EE\autonorm{\nabla_y f(\widetilde{w}_S, \widehat{y}^*(\widetilde{w}_S);\xi)-\nabla_y f(\widetilde{w}_{S'}, \widehat{y}^*(\widetilde{w}_{S'});\xi)}^2\\
            \leq\ &
            2L^2\EE\automedpar{\autonorm{\widetilde{w}_S-\widetilde{w}_{S'}}^2+\autonorm{\widehat{y}^*(\widetilde{w}_S)-\widehat{y}^*(\widetilde{w}_{S'})}^2}\\
            \leq\ &
            2L^2\EE\automedpar{\autopar{1+\autopar{\frac{L+\nu}{\nu}}^2}\autonorm{\widetilde{w}_S-\widetilde{w}_{S'}}^2}\\
            \leq\ &
            2L^2\autopar{1+\autopar{\frac{L+\nu}{\nu}}^2}\cdot\autopar{\frac{2L}{L-\nu}}^2\EE\automedpar{\autonorm{\mathcal{A}_x(S)-\mathcal{A}_x(S')}^2}\\
            \leq\ &
            \frac{8L^4(L+2\nu)^2}{\nu^2(L-\nu)^2}\delta^2,
        \end{split}
    \end{equation}
    where the second inequality comes from \citet[Lemma 4.3]{lin2020gradient}. It concludes that algorithm $\mathcal{A}$ 
    is $\delta$-uniformly primal stable. Applying \citet[Theorem 2]{lei2022stability} to  \eqref{eq:H_1_analysis}, we have
    \begin{equation}
        \begin{split}
            \EE\automedpar{F_S\autopar{\widetilde{w}_S,\widehat{y}_S^*(\widetilde{w}_S)}-F_S\autopar{\widetilde{w}_S,\widehat{y}^*(\widetilde{w}_S)}}
            \leq\ &
            \frac{G}{\nu}\EE\autonorm{\nabla_yF_S\autopar{\widetilde{w}_S, \widehat{y}^*(\widetilde{w}_S)}-\nabla_yF\autopar{\widetilde{w}_S, \widehat{y}^*(\widetilde{w}_S)}}\\
            \leq\ &   
            \frac{G}{\nu}\autopar{4\sqrt{\frac{8L^4(L+2\nu)^2}{\nu^2(L-\nu)^2}}\delta+\sqrt{\frac{\text{Var}(\nabla_y f)}{n}}}\\
            \leq\ &
            \frac{G}{\nu}\autopar{4\sqrt{\frac{8L^4(L+2\nu)^2}{\nu^2(L-\nu)^2}}\delta+\frac{G}{\sqrt{n}}},
        \end{split}
    \end{equation}
    which concludes the proof.
\end{proof}

\begin{lemma}
    \label{lm:NCC_Stability_term2}
    Let $\mathcal{A}$ be an $\delta$-uniformly primal argument stable algorithm. For any function $f$ satisfying Assumption \ref{as:ncsc-ncc} with $\mu=0$, we have
    \begin{equation}
        \begin{split}
            \EE\automedpar{\widehat{\Phi}(w_S)-\widehat{\Phi}_S(w_S)}
            & \leq
            \frac{G}{\nu}\autopar{4\sqrt{\frac{8L^2(L+2\nu)^2}{\nu^2}\autopar{\frac{4L^2}{(L-\nu)^2}\delta^2+\frac{4G^2}{n^2(L-\nu)^2} + \frac{2L^2(G+\nu\sqrt{D_\mathcal{Y}})^2}{n^2\nu^2(L-\nu)^2}}}+\frac{G}{\sqrt{n}}} \\
            & \qquad +
            \frac{G(L+2\nu)}{\nu}\autopar{\frac{2L}{L-\nu}\delta+\frac{2G}{n(L-\nu)} + \frac{2L(G+\nu\sqrt{D_\mathcal{Y}})}{n\nu(L-\nu)}}+ \frac{2G(G+\nu\sqrt{D_\mathcal{Y}})}{n\nu} + \frac{\nu}{2}D_\mathcal{Y}.
        \end{split}
    \end{equation}
\end{lemma}

\begin{proof}
    Note that
    \begin{equation}
        \begin{split}
            &\EE\automedpar{\widehat{\Phi}(w_S)-\widehat{\Phi}_S(w_S)}\\
            =\ &
            \EE\automedpar{F\autopar{w_S,\widehat{y}^*(w_S)}-F_S\autopar{w_S,\widehat{y}_S^*(w_S)}-\frac{\nu}{2}\autonorm{\widehat{y}^*(w_S)}^2+\frac{\nu}{2}\autonorm{\widehat{y}_S^*(w_S)}^2}\\
            \leq\ &
            \EE\automedpar{
            \underbrace{F\autopar{w_S,\widehat{y}^*(w_S)}-F\autopar{w_S,\widehat{y}_S^*(w_S)}}_{J_1}
            +
            \underbrace{F\autopar{w_S,\widehat{y}_S^*(w_S)}-F_S\autopar{w_S,\widehat{y}_S^*(w_S)}}_{J_2}
            }+\frac{\nu}{2}D_\mathcal{Y}.
        \end{split}
    \end{equation}
    For $J_2$, by Lemma \ref{lm:stability_NCC_closeness}, similar to the analysis of $H_2$ in the proof of Lemma~\ref{lm:NCC_Stability_term1}, we have
    \begin{equation}
        \begin{split}
            &\EE\automedpar{f(w_S,\widehat{y}_S^*(w_S);\xi)
            -
            \EE\automedpar{f(w_{S'},\widehat{y}_{S'}^*(w_{S'});\xi)}}\\
            \leq\ &
            G\EE\automedpar{\autonorm{w_S-w_{S'}}+\autonorm{\widehat{y}_S^*(w_S)-\widehat{y}_{S'}^*(w_{S})} + \autonorm{\widehat{y}_{S'}^*(w_S)-\widehat{y}_{S'}^*(w_{S'})}}\\
            \leq\ &
            G\EE\automedpar{\autonorm{w_S-w_{S'}}+\frac{L+\nu}{\nu}\autonorm{w_S-w_{S'}}} + \frac{2G(G+\nu\sqrt{D_\mathcal{Y}})}{n\nu} \\
            \leq\ &
            G\EE\automedpar{\autopar{1+\frac{L+\nu}{\nu}}\cdot \autopar{\frac{2L}{L-\nu}\autonorm{\mathcal{A}_x(S)-\mathcal{A}_x(S')}+\frac{2G}{n(L-\nu)}+\frac{2L(G+\nu D_\mathcal{Y})}{n\nu(L-\nu)}}} + \frac{2G(G+\nu\sqrt{D_\mathcal{Y}})}{n\nu} \\
            \leq\ &
            \frac{G(L+2\nu)}{\nu}\cdot\autopar{\frac{2L}{L-\nu}\EE\automedpar{\autonorm{\mathcal{A}_x(S)-\mathcal{A}_x(S')}}+\frac{2G}{n(L-\nu)}+\frac{2L(G+\nu D_\mathcal{Y})}{n\nu(L-\nu)}} + \frac{2G(G+\nu\sqrt{D_\mathcal{Y}})}{n\nu} \\
            \leq\ &
            \frac{G(L+2\nu)}{\nu}\autopar{\frac{2L}{L-\nu}\delta+\frac{2G}{n(L-\nu)}+\frac{2L(G+\nu D_\mathcal{Y})}{n\nu(L-\nu)}} + \frac{2G(G+\nu\sqrt{D_\mathcal{Y}})}{n\nu}.
        \end{split}
    \end{equation}
    It further holds that
    \begin{equation}
        \EE\automedpar{F\autopar{w_S,\widehat{y}_S^*(w_S)}-F_S\autopar{w_S,\widehat{y}_S^*(w_S)}}
        \leq
        \frac{G(L+2\nu)}{\nu}\autopar{\frac{2L}{L-\nu}\delta+\frac{2G}{n(L-\nu)}+\frac{2L(G+\nu \sqrt{D_\mathcal{Y}})}{n\nu(L-\nu)}} + \frac{2G(G+\nu\sqrt{D_\mathcal{Y}})}{n\nu}.
    \end{equation}
    For $J_1$, similar to the analysis of $H_1$ in the proof of Lemma~\ref{lm:NCC_Stability_term1}, we have
    \begin{equation}
        \begin{split}
            &\EE\autonorm{\nabla_y f(w_S, \widehat{y}^*(w_S);\xi)-\nabla_y f(w_{S'}, \widehat{y}^*(w_{S'});\xi)}^2\\
            \leq\ &
            2L^2\EE\automedpar{\autonorm{w_S-w_{S'}}^2+\autonorm{\widehat{y}^*(w_S)-\widehat{y}^*(w_{S'})}^2}\\
            \leq\ &
            2L^2\EE\automedpar{\autopar{1+\autopar{\frac{L+\nu}{\nu}}^2}\autonorm{w_S-w_{S'}}^2}\\
            \leq\ &
            2L^2\autopar{1+\autopar{\frac{L+\nu}{\nu}}^2}\cdot\EE\automedpar{4\autopar{\frac{2L}{L-\nu}}^2\autonorm{\mathcal{A}_x(S)-\mathcal{A}_x(S')}^2+4\frac{4G^2}{n^2(L-\nu)^2}+2\frac{4L^2(G+\nu \sqrt{D_\mathcal{Y}})^2}{n^2\nu^2(L-\nu)^2}}\\
            \leq\ &
            \frac{8L^2(L+2\nu)^2}{\nu^2}\autopar{\frac{4L^2}{(L-\nu)^2}\delta^2+\frac{4G^2}{n^2(L-\nu)^2}+\frac{2L^2(G+\nu \sqrt{D_\mathcal{Y}})^2}{n^2\nu^2(L-\nu)^2}}.
        \end{split}
    \end{equation}
    Combined with \citet[Theorem 2]{lei2022stability}, we have
    \begin{equation}
        \begin{split}
            &\EE\automedpar{F\autopar{w_S,\widehat{y}_S^*(w_S)}-F\autopar{w_S,\widehat{y}^*(w_S)}}\\
            \leq\ &
            \frac{G}{\nu}\EE\autonorm{\nabla_yF_S\autopar{w_S, \widehat{y}^*(w_S)}-\nabla_yF\autopar{w_S, \widehat{y}^*(w_S)}}\\
            \leq\ &   
            \frac{G}{\nu}\autopar{4\sqrt{\frac{8L^2(L+2\nu)^2}{\nu^2}\autopar{\frac{4L^2}{(L-\nu)^2}\delta^2+\frac{4G^2}{n^2(L-\nu)^2}+\frac{2L^2(G+\nu \sqrt{D_\mathcal{Y}})^2}{n^2\nu^2(L-\nu)^2}}}+\sqrt{\frac{\text{Var}(\nabla_y f)}{n}}}\\
            \leq\ &
            \frac{G}{\nu}\autopar{4\sqrt{\frac{8L^2(L+2\nu)^2}{\nu^2}\autopar{\frac{4L^2}{(L-\nu)^2}\delta^2+\frac{4G^2}{n^2(L-\nu)^2}+\frac{2L^2(G+\nu \sqrt{D_\mathcal{Y}})^2}{n^2\nu^2(L-\nu)^2}}}+\frac{G}{\sqrt{n}}},
        \end{split}
    \end{equation}
    which concludes the proof.
\end{proof}

Next, we formally demonstrate the proof for the generalization bounds in the NC-C setting. 
\begin{theorem}[Stability and Generalization, NC-C, repeat Theorem \ref{thm:stability and generalization NCC}]
    Let $\mathcal{A}$ be an $\delta$-uniformly primal argument stable algorithm, for any function $f$ satisfying Assumption \ref{as:ncsc-ncc} with $\mu=0$, we have
    \begin{equation}
        \mathbb{E}_{\mathcal{A},S}\autonorm{\nabla\Phi^{1/(2L)}\autopar{\mathcal{A}_x(S)}-\nabla\Phi_S^{1/(2L)}\autopar{\mathcal{A}_x(S)}}
        \leq
        \mathcal{O}\autopar{\delta^{\frac{1}{6}}+\autopar{\frac{1}{n}}^{\frac{1}{12}}}.
    \end{equation}
\end{theorem}

\begin{proof}
Recall that
    \begin{equation}
        \nabla\Phi^{1/(2L)}\autopar{\mathcal{A}_x(S)}
        =
        2L\autopar{\mathcal{A}_x(S)-\prox_{\frac{\Phi}{2L}}(\mathcal{A}(S))},
        \quad
        \nabla\Phi_S^{1/(2L)}\autopar{\mathcal{A}_x(S)}
        =
        2L\autopar{\mathcal{A}_x(S)-\prox_{\frac{\Phi_S}{2L}}(\mathcal{A}(S))}.
    \end{equation}
    Since $\Phi$ is $L$-weakly-convex and $G$-Lipschitz \citep[Lemma 4.7]{lin2020gradient},  it holds that
    \begin{equation}
        \autonorm{\nabla\Phi^{1/(2L)}\autopar{\mathcal{A}_x(S)}-\nabla\Phi_S^{1/(2L)}\autopar{\mathcal{A}_x(S)}}
        =
        2L\autonorm{\prox_{\frac{\Phi}{2L}}(\mathcal{A}(S))-\prox_{\frac{\Phi_S}{2L}}(\mathcal{A}(S))}.
    \end{equation}
    Utilizing the regularized objective function, we have
    \begin{equation}
        \label{eq:NCC_decomposition}
        \begin{split}
            &\autonorm{\prox_{\frac{\Phi}{2L}}(\mathcal{A}(S))-\prox_{\frac{\Phi_S}{2L}}(\mathcal{A}(S))}\\
            \leq\ &
            \autonorm{\prox_{\frac{\Phi}{2L}}(\mathcal{A}(S))-\prox_{\frac{\widehat{\Phi}}{2L}}(\mathcal{A}(S))}
            +
            \autonorm{\prox_{\frac{\widehat{\Phi}}{2L}}(\mathcal{A}(S))-\prox_{\frac{\widehat{\Phi}_S}{2L}}(\mathcal{A}(S))}
            +
            \autonorm{\prox_{\frac{\widehat{\Phi}_S}{2L}}(\mathcal{A}(S))-\prox_{\frac{\Phi_S}{2L}}(\mathcal{A}(S))}\\
            \leq\ &
            2\sqrt{\frac{\nu D_\mathcal{Y}}{L-\nu}}+\autonorm{\widetilde{w}_S-w_S},
        \end{split}
    \end{equation}
    where the second inequality comes from Lemma \ref{lemma:nc_sc_to_ncc} with $\lambda=\frac{1}{2L}$.
    So now the problem is transformed to characterizing the distance between $\widetilde{w}_S$ and $w_S$ coming from the regularized surrogate objective which is NC-SC.
    
    Since the function $\widehat{\Phi}(x)+L\autonorm{x-\mathcal{A}(S)}^2$ is $(L-\nu)$-strongly convex, and by the definition of $\widetilde{w}_S$, we have
    \begin{equation}
        \begin{split}
            &\frac{L-\nu}{2}\EE\autonorm{w_S-\widetilde{w}_S}^2\\
            \leq\ &
            \EE\widehat{\Phi}\autopar{w_S}+L\autonorm{w_S-\mathcal{A}(S)}^2
            -\autopar{\widehat{\Phi}\autopar{\widetilde{w}_S}+L\autonorm{\widetilde{w}_S-\mathcal{A}(S)}^2}\\
            =\ &
            \EE\widehat{\Phi}_S\autopar{w_S}+L\autonorm{w_S-\mathcal{A}(S)}^2
            -\autopar{\widehat{\Phi}_S\autopar{\widetilde{w}_S}+L\autonorm{\widetilde{w}_S-\mathcal{A}(S)}^2}
            +
            \autopar{\widehat{\Phi}\autopar{w_S}-\widehat{\Phi}_S\autopar{w_S}}+\autopar{\widehat{\Phi}_S\autopar{\widetilde{w}_S}-\widehat{\Phi}\autopar{\widetilde{w}_S}}\\
            \leq\ &
            \EE\autopar{\widehat{\Phi}\autopar{w_S}-\widehat{\Phi}_S\autopar{w_S}}+\autopar{\widehat{\Phi}_S\autopar{\widetilde{w}_S}-\widehat{\Phi}\autopar{\widetilde{w}_S}}\\
            \leq\ &
            \frac{G}{\nu}\autopar{4\sqrt{\frac{8L^2(L+2\nu)^2}{\nu^2}\autopar{\frac{4L^2}{(L-\nu)^2}\delta^2+\frac{4G^2}{n^2(L-\nu)^2}+\frac{2L^2(G+\nu \sqrt{D_\mathcal{Y}})^2}{n^2\nu^2(L-\nu)^2}}}+\frac{G}{\sqrt{n}}} \\
            & \qquad\quad +
            \frac{G(L+2\nu)}{\nu}\autopar{\frac{2L}{L-\nu}\delta+\frac{2G}{n(L-\nu)} + \frac{2L(G+\nu\sqrt{D_\mathcal{Y}})}{n\nu(L-\nu)}} + \frac{2G(G+\nu\sqrt{D_\mathcal{Y}})}{n\nu} \\
            &\qquad\qquad\qquad +
            \frac{2GL(L+2\nu)}{\nu(L-\nu)}\delta+\frac{G}{\nu}\autopar{4\sqrt{\frac{8L^4(L+2\nu)^2}{\nu^2(L-\nu)^2}}\delta+\frac{G}{\sqrt{n}}}+\nu D_{\mathcal{Y}},
        \end{split}
    \end{equation}
    where the second inequality uses the optimality of $w_S$ and $\widehat{w}_S$, the last inequality is due to Lemma \ref{lm:NCC_Stability_term1} and \ref{lm:NCC_Stability_term2}. Now we choose $\nu$ to simplify the RHS above. For simplicity, first we set $\nu\leq\frac{L}{2}$, so $L-\nu\geq\frac{L}{2}$, $L+2\nu\leq 2L$. The RHS above simplifies to
    \begin{equation}
        \begin{split}
            &\frac{L-\nu}{2}\EE\autonorm{w_S-\widetilde{w}_S}^2\\
            \leq\ &
            \frac{G}{\nu}\autopar{4\sqrt{\frac{32L^4}{\nu^2}\autopar{16\delta^2+\frac{16G^2}{n^2L^2}+\frac{16(G+\nu \sqrt{D_\mathcal{Y}})^2}{n^2\nu^2}}}+\frac{G}{\sqrt{n}}}
            +
            \frac{2GL}{\nu}\autopar{4\delta+\frac{4G}{nL} + \frac{4(G + \nu\sqrt{D_\mathcal{Y}})}{n\nu}} \\
            & \qquad\quad +
            \frac{2G(G + \nu\sqrt{D_\mathcal{Y}})}{n\nu} + \frac{8GL}{\nu}\delta+\frac{G}{\nu}\autopar{4\sqrt{\frac{128L^4}{\nu^2}}\delta+\frac{G}{\sqrt{n}}}+\nu D_{\mathcal{Y}} \\
            \leq\ &
            \frac{G}{\nu}\autopar{\frac{128L^2}{\nu}\sqrt{\delta^2+\frac{G^2}{n^2L^2}+\frac{(G+\nu\sqrt{D_\mathcal{Y}})^2}{n^2\nu^2}}+\frac{G}{\sqrt{n}}}
            +
            \frac{8GL}{\nu}\autopar{2\delta+\frac{G}{nL} + \frac{G+\nu\sqrt{D_\mathcal{Y}}}{n\nu}} \\
            & \qquad\quad +
            \frac{G}{\nu}\autopar{\frac{64L^2}{\nu}\delta + \frac{2(G+\nu\sqrt{D_\mathcal{Y}})}{n} + \frac{G}{\sqrt{n}}}+\nu D_{\mathcal{Y}} \\
            \leq\ &
            \frac{G}{\nu}\autopar{\frac{128L^2}{\nu}\autopar{\delta+\frac{G}{nL}+\frac{G+\nu \sqrt{D_\mathcal{Y}}}{n\nu}}+\frac{G}{\sqrt{n}}} +
            \frac{8GL}{\nu}\autopar{2\delta+\frac{G}{nL} + \frac{G+\nu\sqrt{D_\mathcal{Y}}}{n\nu}} \\
            & \qquad\quad +
            \frac{G}{\nu}\autopar{\frac{64L^2}{\nu}\delta + \frac{2(G+\nu\sqrt{D_\mathcal{Y}})}{n} + \frac{G}{\sqrt{n}}}+\nu D_{\mathcal{Y}} \\
            =\ &
            64G\autopar{3\delta+\frac{2G}{nL}+\frac{2G+2\nu\sqrt{D_\mathcal{Y}}}{n\nu}}\frac{L^2}{\nu^2} +
            2G\autopar{
                8\delta + \frac{G}{\sqrt{n}L} + \frac{4G}{nL} + \frac{G+\nu\sqrt{D_\mathcal{Y}}}{n}\autopar{\frac{4}{\nu} + \frac{1}{L}}
            }
            \frac{L}{\nu} + \nu D_{\mathcal{Y}} \\
            =\ &
            \mathcal{O}\autopar{\frac{1}{n}}\cdot\mathcal{O}\autopar{\frac{1}{\nu^3}+\frac{1}{\nu^2}+\frac{1}{\nu} + 1}
            +
            \mathcal{O}\autopar{\frac{1}{\sqrt{n}}}\cdot\mathcal{O}\autopar{\frac{1}{\nu}}
            +
            \mathcal{O}\autopar{\delta}\cdot\mathcal{O}\autopar{\frac{1}{\nu^2}+\frac{1}{\nu}}+
            \mathcal{O}(1)\nu,
        \end{split}
    \end{equation}
    where the last step hides all other dependence on parameters except $\delta$ and $n$. Let
    \begin{equation}
        \frac{1}{\nu}=\mathcal{O}\autopar{\min\autopar{\delta^{-\frac{1}{3}},n^{\frac{1}{4}}}},
    \end{equation}
    with $\delta$ and $1/n$ small enough such that $\nu\leq L/2$ holds. The setting of $\nu$ implies that
    \begin{equation}
        \begin{split}
            \EE\autonorm{w_S-\widetilde{w}_S}^2
            \leq\ &
            \mathcal{O}\autopar{\frac{1}{n}}\cdot\mathcal{O}\autopar{\frac{1}{\nu^3}}
            +
            \mathcal{O}\autopar{\frac{1}{\sqrt{n}}}\cdot\mathcal{O}\autopar{\frac{1}{\nu}}+
            \mathcal{O}\autopar{\delta}\cdot\mathcal{O}\autopar{\frac{1}{\nu^2}}+
            \mathcal{O}(1)\nu\\
            \leq\ &
            \mathcal{O}\autopar{\frac{1}{n}}\cdot\mathcal{O}\autopar{n^{\frac{3}{4}}}+\mathcal{O}\autopar{\frac{1}{\sqrt{n}}}\cdot\mathcal{O}\autopar{n^{\frac{1}{4}}}+\mathcal{O}\autopar{\delta}\cdot\mathcal{O}\autopar{\delta^{-\frac{2}{3}}}
            +
            \mathcal{O}\autopar{\delta^{\frac{1}{3}}+n^{-\frac{1}{4}}}\\
            =\ &
            \mathcal{O}\autopar{\delta^{\frac{1}{3}}+n^{-\frac{1}{4}}}.
        \end{split}
    \end{equation}    
    As a result, we have
    \begin{equation}
        \EE\autonorm{w_S-\widetilde{w}_S}\leq\mathcal{O}\autopar{\delta^{\frac{1}{6}}+n^{-\frac{1}{8}}}.
    \end{equation}
    Further incorporating \eqref{eq:NCC_decomposition}, we have
    \begin{equation}
        \begin{split}
            &\mathbb{E}_{\mathcal{A},S}\autonorm{\nabla\Phi^{1/(2L)}\autopar{\mathcal{A}_x(S)}-\nabla\Phi_S^{1/(2L)}\autopar{\mathcal{A}_x(S)}}\\
            \leq\ &
            2\sqrt{\frac{\nu D_\mathcal{Y}}{L-\nu}}+\EE\autonorm{\widetilde{w}_S-w_S}\\
            \leq\ &
            \mathcal{O}\autopar{\sqrt{\nu}}+\EE\autonorm{\widetilde{w}_S-w_S}\\
            =\ &
            \mathcal{O}\autopar{\delta^{\frac{1}{6}}+n^{-\frac{1}{8}}},
        \end{split}
    \end{equation}
    which concludes the proof.
\end{proof}

\section{Proof of Corollary \ref{thm:NC-SC SGDA Specific} and \ref{thm:NC-C SGDA Specific}}
\label{apdx:NCSC SGDA}

\begin{corollary}
    Assume the function $f$ is NC-SC as defined in Assumption \ref{as:ncsc-ncc}, then if we run SGDA for $T$ iterations with stepsize $\autopar{\alpha_x,\alpha_y}=\autopar{\frac{c}{t},\frac{cr^2}{t}}$ for some constant $c>0$ and $1\leq r<\kappa$, we have 
    \begin{equation}
        \begin{split}
            \EE_{S,\mathcal{A}}\autonorm{\nabla\Phi\autopar{\mathcal{A}_x(S)}-\nabla\Phi_S\autopar{\mathcal{A}_x(S)}}
            \leq
            \autopar{1+\kappa}
            \autopar{\frac{8G\autopar{1+\frac{1}{cL(r+1)}}}{n}\autopar{24\kappa cL(r+1)}^{\frac{1}{cL(r+1)+1}}T^{\frac{cL(r+1)}{cL(r+1)+1}}+\frac{G}{\sqrt{n}}}.
        \end{split}
    \end{equation}
\end{corollary}

\begin{proof}
    Denote $\Delta_t\triangleq\sqrt{\autonorm{x_t-x_t'}^2+\autonorm{y_t-y_t'}^2}$, and the event $E_{t_0}=\mathbf{1}\autopar{\Delta_{t_0}=0}$, we have for the full gradient $\nabla f=(\nabla_x f, \nabla_y f)^\top$,
    \begin{equation}
        \begin{split}
            & \mathbb{E}\autonorm{\nabla f(x_t,y^*(x_t);\xi)-\nabla f(x_t',y^*(x_t');\xi)}\\
            =\ &
            \mathbb{P}(E_{t_0})\mathbb{E}\automedpar{\autonorm{\nabla f(x_t,y^*(x_t);\xi)-\nabla f(x_t',y^*(x_t');\xi)}|E_{t_0}}\\
            &\qquad\qquad\qquad\qquad+
            \mathbb{P}(E_{t_0}^C)\mathbb{E}\automedpar{\autonorm{\nabla f(x_t,y^*(x_t);\xi)-\nabla f(x_t',y^*(x_t');\xi)}|E_{t_0}^C}\\
            \leq\ &
            \mathbb{E}\automedpar{\autonorm{\nabla f(x_t,y^*(x_t);\xi)-\nabla f(x_t',y^*(x_t');\xi)}|E_{t_0}}+2G\mathbb{P}(E_{t_0}^C)\\
            \leq\ &
            \mathbb{E}\automedpar{\autonorm{\nabla_x f(x_t,y^*(x_t);\xi)-\nabla_x f(x_t',y^*(x_t');\xi)}+\autonorm{\nabla_y f(x_t,y^*(x_t);\xi)-\nabla_y f(x_t',y^*(x_t');\xi)}|E_{t_0}}+2G\mathbb{P}(E_{t_0}^C)\\
            \leq\ &
            2L\mathbb{E}\automedpar{\autonorm{x_t-x_t'}+\autonorm{y^*(x_t)-y^*(x_t')}|E_{t_0}}+2G\mathbb{P}(E_{t_0}^C)\\
            \leq\ &
            2(1+\kappa)L\mathbb{E}\automedpar{\autonorm{x_t-x_t'}|E_{t_0}}+2G\frac{t_0}{n}\\
            \leq\ &
            4\kappa L\mathbb{E}\automedpar{\Delta_t|\Delta_{t_0}=0}+2G\frac{t_0}{n},
        \end{split}
    \end{equation}
    the remaining steps aims to bound $\mathbb{E}\automedpar{\Delta_t|\Delta_{t_0}=0}$, which are the same as those in \citep[Appendix B.8]{farnia2021train}, with that we will get
    \begin{equation}
        \mathbb{E}\autonorm{\nabla f(x_T,y^*(x_T);\xi)-\nabla f(x_T',y^*(x_T');\xi)}
        \leq
        \frac{4\kappa L \cdot 12 G}{nL}\autopar{\frac{T}{t_0}}^{cL(r+1)}+\frac{2G}{n}t_0,
    \end{equation}
    to minimize the RHS above over $t_0$, we set
    \begin{equation}
        t_0=\autopar{\frac{\frac{4\kappa L \cdot 12 G}{nL}\cdot cL(r+1)}{\frac{2G}{n}}}^{\frac{1}{cL(r+1)+1}}\cdot T^{\frac{cL(r+1)}{cL(r+1)+1}}
        =
        \autopar{24\kappa cL(r+1)}^{\frac{1}{cL(r+1)+1}}T^{\frac{cL(r+1)}{cL(r+1)+1}}
    \end{equation}
    and we get
    \begin{equation}
        \mathbb{E}\autonorm{\nabla f(x_T,y^*(x_T);\xi)-\nabla f(x_T',y^*(x_T');\xi)}
        \leq
        \frac{2G\autopar{1+\frac{1}{cL(r+1)}}}{n}\autopar{24\kappa cL(r+1)}^{\frac{1}{cL(r+1)+1}}T^{\frac{cL(r+1)}{cL(r+1)+1}}.
    \end{equation}
    We conclude the proof by incorporating the above bound with Theorem \ref{thm:stability and generalization}.
\end{proof}

\begin{corollary}
    Assume the function $f$ is NC-C as defined in Assumption \ref{as:ncsc-ncc} with $\mu=0$, then if we run SGDA for $T$ iterations with stepsize
    $\max\{\alpha_x,\alpha_y\}\leq\frac{c}{t}$ for some constant $c>0$, we have 
    \begin{equation}
        \begin{split}
            \EE_{S,\mathcal{A}} \autonorm{\nabla\Phi^{1/(2L)}\autopar{\mathcal{A}_x(S)}-\nabla\Phi_S^{1/(2L)}\autopar{\mathcal{A}_x(S)}}
            \leq
            \mathcal{O}\autopar{\autopar{\frac{T^\frac{cL}{cL+1}}{n}}^{1/6} + \autopar{\frac{1}{n}}^{1/8}}.
        \end{split}
    \end{equation}
\end{corollary}

\begin{proof}
    Denote $\Delta_t\triangleq\sqrt{\autonorm{x_t-x_t'}^2+\autonorm{y_t-y_t'}^2}$, and the event $E_{t_0}=\mathbf{1}\autopar{\Delta_{t_0}=0}$, we have that
    \begin{equation}
        \begin{split}        
            \mathbb{E}\| x_t - x_t'\|
            & \leq
            \EE\Delta_t \\
            & =
            \mathbb{P}(E_{t_0})\mathbb{E}\automedpar{
            \Delta_t|E_{t_0}} +
            \mathbb{P}(E_{t_0}^C)\mathbb{E}\automedpar{
            \Delta_t|E_{t_0}^C} \\
            & \leq
            \mathbb{E}\automedpar{\Delta_t|E_{t_0}}+2\sqrt{D_\mathcal{X} + D_\mathcal{Y}}\mathbb{P}(E_{t_0}^C) \\
            & \leq
            \mathbb{E}\automedpar{\Delta_t|\Delta_{t_0}=0} + 2\sqrt{D_\mathcal{X} + D_\mathcal{Y}}\frac{t_0}{n},
        \end{split}
    \end{equation}
    the remaining steps aims to bound $\mathbb{E}\automedpar{\Delta_t|\Delta_{t_0}=0}$, with the results in \citep[Appendix B.9]{farnia2021train}, we get
    \begin{equation}
        \mathbb{E}\autonorm{x_T-x_T'}
        \leq
        \frac{2G}{nL}\autopar{\frac{T}{t_0}}^{cL} + 2\sqrt{D}\frac{t_0}{n},
    \end{equation}
    where we let $D=D_\mathcal{X} + D_\mathcal{Y}$. To minimize the RHS above over $t_0$, we set
    \begin{equation}
        t_0 = \autopar{\frac{cG}{\sqrt{D}}}^{1/(cL+1)}T^\frac{cL}{cL+1},
    \end{equation}
    and we get
    \begin{equation}
        \mathbb{E}\autonorm{x_T-x_T'}
        \leq
        2\autopar{\frac{G}{L}\autopar{\frac{1}{cG}}^\frac{cL}{cL+1} + (cG)^\frac{1}{cL+1}}D^\frac{cL}{2(cL+1)}\frac{T^\frac{cL}{cL+1}}{n}.
    \end{equation}
    The proof is complete by incorporating the above bound with Theorem \ref{thm:stability and generalization NCC}.
\end{proof}

\section{Proof of Corollary \ref{thm:NC-SC SDA} and \ref{thm:NC-C SDA}}
\label{apdx:NCC SGDA}
\begin{proof}
    For the NC-SC case, by \citet[Corollary 6]{lei2022stability}, we know the algorithm is $\delta$ uniformly primal stable in gradients with $\delta=2G\sqrt{T/n}$, the proof is complete by Theorem \ref{thm:stability and generalization}.
    
    For the NC-C case, we want to derive the uniform primal argument stability, the flow here is almost the same as the proof of \citet[Corollary 6]{lei2022stability}, let $\Omega_t=\autonorm{x_t-x_t'}^2$, define the event $E_\Omega$ as that the only different data point $\xi_i$ is selected by the algorithm $\mathcal{A}$, so we have
\begin{equation}
    \EE\automedpar{\Omega_t}\leq \EE\automedpar{\Omega_t~|~E_\Omega}P(E_\Omega) 
    +
    \EE\automedpar{\Omega_t~|~E_\Omega^C}\mathbb{P}(E_\Omega^C)
    \leq
    \EE\automedpar{\Omega_t~|~E_\Omega^C}\frac{T}{n}
    \leq\frac{4D_\mathcal{X}T}{n},
\end{equation}
so the algorithm is $\sqrt{4D_\mathcal{X}T/n}$-uniformly primal argument stable. Then we conclude the proof by substituting the above stability results into Theorem \ref{thm:stability and generalization NCC}, i.e.,
\begin{equation}
    \begin{split}
        &\EE_{S,\mathcal{A}}\autonorm{\nabla\Phi^{1/(2L)}\autopar{\mathcal{A}_x(S)}-\nabla\Phi_S^{1/(2L)}\autopar{\mathcal{A}_x(S)}}\\
    =\ &
    \mathcal{O}\autopar{\delta^{\frac{1}{6}}+n^{-\frac{1}{8}}}
    =
    \mathcal{O}\autopar{\autopar{\frac{T}{n}}^{\frac{1}{12}}+\autopar{\frac{ 1}{n}}^{\frac{1}{8}}}
    =
    \mathcal{O}\autopar{\autopar{\frac{T}{n}}^{\frac{1}{12}}+\autopar{\frac{ 1}{n}}^{\frac{1}{8}}}.
    \end{split}
\end{equation}
\end{proof}

\end{document}